\documentclass[11pt, reqno]{amsart}

\usepackage[english]{babel}

\usepackage[left=3.5cm, right=3.5cm, top=3.7cm, bottom=3.9cm]{geometry}

\usepackage{url, amsfonts, amsmath, amsthm, amssymb, mathtools, mathrsfs, enumerate}
\usepackage{color, xcolor, verbatim, tensor, tikz, tikz-cd}
\usetikzlibrary{angles,quotes}
\usepackage[colorlinks=true,citecolor=blue]{hyperref}
\usetikzlibrary{matrix,calc}
\definecolor{wine-stain}{rgb}{0.5,0,0}

\newcommand{\red}[1]{\textcolor{red}{#1}}

\newtheorem{thm}{Theorem}[section]
\newtheorem{prop}[thm]{Proposition}
\newtheorem{lem}[thm]{Lemma}

\theoremstyle{definition}
\newtheorem{defn}[thm]{Definition}
\newtheorem{rem}[thm]{Remark}
\newtheorem{claim}{Claim}

\newtheorem*{ques*}{Question}
\newtheorem*{thm*}{Theorem}
\newtheorem*{rem*}{Remark}
\newtheorem*{rems*}{Remarks}
\newtheorem*{exs*}{Examples}
\newtheorem*{mthm*}{Main Theorem}

\numberwithin{equation}{section}

\newcommand{\al}{\alpha}
\newcommand{\be}{\beta}
\newcommand{\ga}{\gamma}
\newcommand{\del}{\delta}
\newcommand{\la}{\lambda}

\newcommand{\om}{\omega}

\newcommand{\vp}{\varphi}

\newcommand{\tpsi}{\tilde\psi}

\newcommand{\htheta}{\hat\theta}

\renewcommand{\d}{\partial}
\newcommand{\dbar}{\overline{\partial}}
\newcommand{\ddbar}{\sqrt{-1}\d\overline{\d}}

\newcommand{\ii}{\sqrt{-1}}

\newcommand{\NN}{\mathbb{N}}

\newcommand{\RR}{\mathbb{R}}
\newcommand{\CC}{\mathbb{C}}

\newcommand{\PP}{\mathbb{P}}

\newcommand{\sR}{\mathcal{R}}
\newcommand{\sS}{\mathcal{S}}

\newcommand{\im}{\mathrm{Im}}
\newcommand{\re}{\mathrm{Re}}

\newcommand{\Bl}{\mathrm{Bl}}

\newcommand{\arccot}{\mathrm{arccot}}

\makeatletter
\renewcommand*{\eqref}[1]{%
	\hyperref[{#1}]{\textup{\tagform@{\ref*{#1}}}}%
}
\makeatother

\title[Singularity formation of cotangent flow on $\text{Bl}_{x_0}\mathbb{C}\mathbb{P}^{3}$]{Singularity formation in co-dimension one of the dHYM cotangent flow on blow up of $\mathbb{C}\mathbb{P}^{3}$ at a point}
\author[Ramesh Mete]{Ramesh Mete}
\address{Department of Mathematics, Indian Institute of Technology Bombay, Powai, Mumbai - 400076, Maharashtra, India.}
\email{rameshm@math.iitb.ac.in, ramesh2025m@gmail.com}

\subjclass[2020]{Primary 32Q15; Secondary 53C55, 58J35, 35A21}
\keywords{dHYM equation, dHYM instability, cotangent flow, blow up of complex projective space at a point, Calabi ansatz}
\thanks{Work supported in part by a PhD fellowship from the Indian Institute of Science, Bengaluru}

\begin{document}
	
\begin{abstract}
The existence and uniqueness of canonical singular solutions of the J equation and the deformed Hermitian Yang Mills (dHYM) equation was proved in \cite{DMS24} on compact K\"ahler surfaces. In this paper, we study the singularity formation of the dHYM cotangent flow on the one-point blow up of $\CC\PP^3$ using Calabi ansatz. In particular, we provide an explicit example where the flow develops a singularity along the exceptional divisor. Moreover, the limit satisfies corresponding singular dHYM equation in the sense of \cite{DMS24} and provides some evidence for Conjecture $1.12$ in \cite{DMS24} on this three dimensional manifold with symmetry.
\end{abstract}
	
\maketitle
	
	
\section{\large Introduction}
\label{section:introduction to dHYM}

The deformed Hermitian Yang Mills equation is a special type of complex Hessian equation which first appeared in \cite{MMMG00, LYZ01} from different view-points. It has connections to mirror symmetry in string theory (see \cite{CXY18} more details). In \cite{JY}, Jacob-Yau initiated the systematic study of the equation from a mathematical view-point. Recently, it has received considerable attention - for instance, see \cite{JY, CJY, CXY18, gchen, DatPin2021, Pingali-dHYM 3 fold, Chu-Lee-Tak, JacSh, Bal23, FYZ, ChaJac23, DMS24} and references therein.

\vspace*{1mm}
Let $(X,\om)$ be a compact K\"ahler manifold of complex dimension $n$. Fix a real $(1,1)$ cohomology class $\al\in H^{1,1}(X,\RR):= H_{{\rm dR}}^{2}(X, \RR)\cap H_{\dbar}^{1,1}(X, \CC)$, and denote the K\"ahler class of $\om$ by $\be:=[\om]$. Then the {\em deformed Hermitian Yang Mills} (dHYM) equation is given by
\begin{equation}\label{eq:DHYM}
\re(\chi + \ii \om)^n = \cot(\hat\theta) \im(\chi + \ii \om)^n,
\end{equation}
where the desired solution $\chi$ is a smooth closed real $(1,1)$-form in $\al$, and $\htheta$ is a constant satisfying the following
\begin{equation}\label{eq:basic cohomological condition}
\re(\al+\ii\be)^n-(\cot\htheta)\im(\al+ \ii\be)^n=0.
\end{equation}
Here and for rest of the paper, $(\al+\ii\be)^n := \int_{X} (\chi + \ii \om)^n$ for any representatives $\om\in\be$, $\chi\in\al$, and it depends only on the cohomology classes $\al$ and $\be$.

\vspace*{1mm}
Assuming $(\al+\ii\be)^n\neq 0$, note that $\htheta$ is well-defined mod $2\pi$. We will assume throughout that $\htheta$ is normalized to be in the {\em supercritical phase}, i.e. $\htheta\in(0,\pi)$. In particular, we are also assuming that $\im(\al+\ii\be)^n>0$. We now recall a notion of slope stability.
	
\begin{defn}
\label{def:stability of the triple X alpha be}
Assume $\htheta\in(0,\pi)$ satisfying \eqref{eq:basic cohomological condition}. The triple $(X, \alpha, \beta)$ is said to be
\begin{enumerate}
\item {\em dHYM semi-stable} (resp. {\em stable}) if there exists a K\"ahler class $\gamma$ such that 
\begin{equation}\label{eq:numerical cond-dHYM sol-ChuLeeTaka-gen Kah mfd}
\big(\mathrm{Re}(\al+\ii\be)^k-\cot\htheta\cdot\mathrm{Im}(\al+\ii\be)^k\big)\cdot\ga^{m-k}\cdot Y\geq 0
\end{equation}
for any analytic subvariety $Y$ of dimension $m$ and any integer $k\in[1,m]$ (resp. strict inequality in \eqref{eq:numerical cond-dHYM sol-ChuLeeTaka-gen Kah mfd} when $m<n$).
			
\item {\em dHYM unstable} if it is not dHYM semi-stable.
\end{enumerate}
\end{defn}
	
In view of Definition \ref{def:stability of the triple X alpha be} and based on the results in \cite{gchen, Song2020}, Chu-Lee-Takahashi \cite{Chu-Lee-Tak} showed that there exists a (unique) solution of \eqref{eq:DHYM} with $\htheta\in(0,\pi)$ if and only if $(X,\al,\be)$ is dHYM stable. When $X$ is projective it is enough to check that
\begin{equation}\label{eq:numerical cond-dHYM-CJY conjecture}
\big(\mathrm{Re}(\al+\ii\be)^m-\cot\htheta\cdot\mathrm{Im}(\al+\ii\be)^m\big)\cdot Y > 0
\end{equation}
for any analytic subvariety $Y\subsetneq X$ with $0< m:=\dim_{\CC} Y < n$ (cf. \cite{Chu-Lee-Tak}). This second result confirms Conjecture 1.5 in \cite{CJY} in the projective case, which was also proved in \cite{Bal23} building on the works by Datar-Pingali \cite{DatPin2021}. Therefore, for projective manifolds we have that $(X,\al,\be)$ is dHYM stable (resp. dHYM semi-stable) if and only if the condition \eqref{eq:numerical cond-dHYM-CJY conjecture} holds (resp. the left hand side of \eqref{eq:numerical cond-dHYM-CJY conjecture} is non-negative) for any proper analytic subvariety $Y$ of $X$ with $0<m:=\dim_{\CC} Y<n$.

\vspace*{1mm}
If $\{\la_i\}_{i=1}^{n}$ are eigenvalues of $\om^{-1}\chi\in \text{End}(T^{1,0}X)$, then the dHYM equation \eqref{eq:DHYM} is equivalent to the following (cf. \cite{JY})
\begin{equation}
\label{eq:dHYM:PDE version}
\theta_{\om}(\chi) := \sum_{i=1}^{n}\arccot(\la_i) =\hat\theta,
\end{equation} 
where $\arccot(\la):=\pi/2 -\arctan(\la) \in (0,\pi)$ for any $\la\in\RR$. The map $\theta_\om$ is known as the {\em Lagrangian phase operator}. Let us fix a representative $\chi\in\al$. The (dHYM) {\em cotangent flow}, introduced in \cite{FYZ}, is defined by
\begin{equation}\label{eq:dHYM cotangent flow}
\begin{cases}
\frac{\d\vp}{\d t} =\cot \theta_\om(\chi_{\vp})-\cot\hat\theta, \\
\vp|_{t=0} = \vp_0,
\end{cases}
\end{equation}
where $\chi_\vp := \chi + \ddbar\vp$ for $\vp\equiv\vp(t)$ and $\hat\theta\in(0,\pi)$ satisfying \eqref{eq:basic cohomological condition}. They showed that the cotangent flow exists for all positive time if the initial function $\vp_0\in C^\infty(X,\RR)$ satisfies $\theta_\om(\chi_{\vp_0}) < \pi$ (cf. \cite[Theorem 1.2]{FYZ}).

	\vspace*{1mm}
	In this paper, we study the equation \eqref{eq:DHYM}  on $X=\Bl_{x_0}\PP^3$, the blow up of the complex projective space $\CC\PP^3$ at a point, in the dHYM unstable case. Our main result is the following.
	
	\begin{thm}\label{thm:singular-dhym-P3 blowup-cotangnet flow_Introduction}
		Let $X = \Bl_{x_0}\PP^3$ be the blow up of $\CC\PP^3$ at a point with a given K\"ahler class $\be:= b[H]-[E]$ and a real $(1,1)$ class $\al:= p[H]-q[E]$, where $b>1$ and both $p, q>0$. Fix a K\"ahler form $\om\in\be$ and a smooth representative $\chi\in\al$ both satisfying Calabi ansatz (see $\S$ \ref{subsec:Calabi-symm-reduction}). Let $\htheta\in(0,\pi)$ such that \eqref{eq:basic cohomological condition} holds, and we set $$c_q:= \cot\htheta$$ (see \eqref{eq:constant c_q for blow up of CP3 at a point} for the explicit value of $c_q$ in terms of $b,p,q$). Assume that there exists a smooth real $(1,1)$-form $\chi_0\in\al$ satisfying Calabi ansatz such that the corresponding Calabi potential $\psi_0$ (see $\S$ \ref{subsec:Calabi-symm-reduction}) satisfies the following:
		\begin{itemize}
			\item $\theta_{\psi_0}<\pi$, where the function $\theta_{\psi_0}:[1,b]\to (0,n\pi)$ is given by \eqref{eq:Lag phase func with Calabi ansatz on blow up of CPn at a point} with $n=3$ here and nothing but the Lagrangian phase $\theta_\om(\chi_0)$ defined above. Note that this condition implies the long-time existence of the dHYM cotangent flow \eqref{eq:dHYM cotangent flow} according to \cite[Theorem 1.2]{FYZ}.
			
			\item and the initial function $\psi_0$ satisfies the condition 
			\begin{equation}\tag{\textbf{H1}}
				\label{eq:comparison and monotonicity for initial function_Introduction}
				\psi_0 \leq \tpsi_\xi ~~\text{and}~~ (\cot \theta_{\psi_0})' > 0,
			\end{equation}
			where the function $\tilde\psi_{\xi}:[1,b]\longrightarrow\RR$ is defined in $\S$ \ref{section:aux family on P n blowup} with $\tilde\psi_{\xi}(1) = \xi$ and $\tilde\psi_{\xi}(b) = p$, and the constant $\xi\in [q, \xi')$ is the unique root of the polynomial $F$ given by \eqref{eq:poly-F-in-dim-3-explicit-form} with the constant $\xi'$ is defined in \eqref{eq:xi prime defn for blowup pf CP3 at a point}.
		\end{itemize}
		Then the cotangent flow \eqref{eq:dHYM cotangent flow} $\chi_t:= \chi + \ddbar\vp(t)$ (emanating from $\chi_0$) exists for all time and has the following convergence behaviour:
		\begin{enumerate}[~(1)]
			\item If $q > c_q + \sqrt{c_q^2 + 1}$, then the triple $(X,\al,\be)$ is dHYM stable and the cotangent flow $\chi_t$ converges smoothly on $X$ to the solution $\chi_\infty$ of the dHYM equation
			\begin{equation}\label{eq:DHYM-P3-blowup_Introduction}
				\re(\chi_\infty + \ii \om)^3 = c_q \im(\chi_\infty + \ii\om)^3.
			\end{equation}
			\vspace*{0.1mm}
			
			\item If $q = c_q + \sqrt{c_q^2 + 1}$, then $(X,\al,\be)$ is dHYM semistable but not stable. In this case, $\chi_t$ converges to $\chi_\infty$ (say) in the sense of $(1,1)$-currents on $X$, where $\chi_\infty$ has bounded local potential and the convergence is smooth on $X\setminus E$. Moreover, $\chi_\infty$ solves the equation \eqref{eq:DHYM-P3-blowup_Introduction} smoothly on $X\setminus E$, and it solves the same equation on $X$ but now the wedge product is taken in the sense of Bedford-Taylor \cite{BedTay87}. \vspace*{3mm}
			
			\item If $0< q < c_q + \sqrt{c_q^2 + 1}$, then $(X,\al,\be)$ is dHYM unstable. Here also, the cotangent flow $\chi_t$ converges to $\chi_\infty$ (say) in the sense of $(1,1)$-currents on $X$ (and smoothly on $X\setminus E$) but $\chi_\infty$ does not have bounded local potential. Instead, $\chi_\infty$ can be decomposed as $$\chi_\infty = \chi'_\infty + (\xi -q)[E],$$ where $\chi'_{\infty}\in p[H]-\xi[E]$ is a closed $(1,1)$-current with continuous local potential. Note that $\xi > q$ in the dHYM unstable case. Furthermore, $\chi'_\infty$ solves the following dHYM equation on $X\setminus E$
			\begin{equation*}
				\re(\chi'_\infty + \ii \om)^3 = \zeta \cdot \im(\chi'_\infty + \ii \om)^3,
			\end{equation*}
			where the new dHYM-slope $\zeta > c_{q}$ is given by
			\begin{align}\label{eq:new dHYM slope zeta for CP3 blowup at a point}
				\zeta := \sup \left\{ \frac{\re(\al_s + \ii \be)^3}{\im(\al_s + \ii \be)^3} : ~ \al_s= p[H]-s[E], ~~ s\in \sS_{3}^{+}\cap(0,\infty) \right\},
			\end{align}
			where the subset $\sS_{3}^{+}\subset\RR$ is defined in $\S$ \ref{section:aux family on P n blowup}.
		\end{enumerate}
	\end{thm}
	
	A (normalized) K\"ahler class on $X = \Bl_{x_0}\PP^n$ is of the form $\be=b[H]-[E]$ for $b>1$ and a general real $(1,1)$-cohomology class has the form $\al=p[H]-q[E]$ for any $p,q\in\RR$ (see $\S$ \ref{subsec:Calabi-symm-reduction}). We will see in Lemma \ref{lem:dHYM stab ineqs for CP3 blowup at a point} that the dHYM stability (resp. dHYM semi-stability) of the triple $(X,\al,\be)$ with $X=\Bl_{x_0}\PP^3$, $\al=p[H]-q[E]$ and $\be=b[H]-[E]$ is equivalent to 
	\begin{equation*}
		p > b\left(c_q + \sqrt{1+c_q^2}\right),~~ q > \left(c_q + \sqrt{1+c_q^2}\right)~~~(\text{resp.}~ \geq),
	\end{equation*}
	(see \eqref{eq:dim 3 dHYM stability-blow up P 3 at one point}). This is similar to the surface case $X=\Bl_{x_0}\PP^2$ where the dHYM stability is equivalent to the condition \eqref{eq:dim 2 dHYM stability-blow up P 3 at one point}, and the first inequality $p > b c_q$ in \eqref{eq:dim 2 dHYM stability-blow up P 3 at one point} always holds for the supercritical phase $\hat{\theta}\in(0,\pi)$ (cf. \cite{DMS24}). Analogously in Lemma \ref{lem:inequality involving p and b holds for blowup of P 3} we will see that the first inequality in \eqref{eq:dim 3 dHYM stability-blow up P 3 at one point} always holds on $X=\Bl_{x_0}\PP^3$ for the supercritical phase $\hat{\theta}\in(0,\pi)$ and for $p,q > 0$. In fact, if we further assume that $q \geq p/b$, the other inequality in \eqref{eq:dim 3 dHYM stability-blow up P 3 at one point} also holds and hence $(X, \al, \be)$ is dHYM stable (see Remark \ref{rem:only-possibility-of-dhym-unstab}). Therefore, in order to study the dHYM instability of $(X,\al,\be)$ (with $p,q > 0$) we must assume that $0 < q < p/b$, i.e. in particular, $\al$ is also a K\"ahler class.

	\begin{rem}
		As we will see in $\S$ \ref{sec:proof-main theorem-P3-blowup}, the key point is to find an initial data $\chi_0\in\al$ (or equivalently, an initial Calabi potential $\psi_0$) which satisfies certain properties to apply the dHYM cotangent flow method. For $X=\Bl_{x_0}\PP^3$, we have a choice of initial function $\psi_0$ given by \eqref{eq:initial function-dim 3} which satisfies all but the first condition of \eqref{eq:comparison and monotonicity for initial function_Introduction} mentioned in Theorem \ref{thm:singular-dhym-P3 blowup-cotangnet flow_Introduction} for arbitrary $b,p,q$ with $b>1$ and $0<q<p_{\star}$ (see $\S$ \ref{section:aux family on P n blowup} for the constant $p_{\star}$). Further, the same initial function $\psi_0$ given by \eqref{eq:initial function-dim 3} with particular values $p=2bq=18$ and $q=b=3$ also satisfies the first condition of \eqref{eq:comparison and monotonicity for initial function_Introduction} (see Lemma \ref{lem:initial-func-less-than-limit}). We also expect the same to be true for other values of $p,b,q$.
	\end{rem}
	
	\begin{rem}
		As we have discussed above the triple $(X,\al,\be)$ is dHYM unstable if and only if $0< q < c_q + \sqrt{1 + c_{q}^2}$. In $\S$ \ref{subsec:examples-dhym-unstab-types}, we divide this dHYM instability inequality into two parts -- (i) $c_q < q < c_q + \sqrt{1 + c_{q}^2}$, which we will call as ``{\em first kind of dHYM instability}", where dHYM instability occurs along the exceptional divisor $E$; and (ii) $0< q \leq c_q$, which we will call as ``{\em second kind of dHYM instability}", where the dHYM instability happens along co-dimension two subvarieties. The specific example $p=2bq=18$ with $q=b=3$ gives us an example of ``{\em first kind of dHYM instability}". It would be interesting to study the singularity formation of the flow for ``{\em second kind of dHYM instability}".
	\end{rem}
	
	\begin{rem}\label{rem:explicit values of c_q 4th deg poly and zeta}
		The constant $\zeta$ in Theorem \ref{thm:singular-dhym-P3 blowup-cotangnet flow_Introduction} is equal to $c_{\xi}$, where $\xi$ is the unique solution of the polynomial $F$ on $(q,\xi')$ (see $\S$ \ref{section:aux family on P n blowup}). Moreover, the constant $\zeta$ is the invariant version of the constant $\cot\theta_{\min}$ from Conjecture $1.12$ in \cite{DMS24} in the unstable case. We expect that the $(1,1)$-current $\chi_\infty$ in part $(3)$ of Theorem \ref{thm:singular-dhym-P3 blowup-cotangnet flow_Introduction}  solves the following ``singular" dHYM equation
		\begin{equation}
			\re\langle(\chi_\infty + \ii \om)^3\rangle = \zeta \cdot \im\langle(\chi_\infty + \ii \om)^3\rangle,
		\end{equation}
		on $X = \Bl_{x_0}\PP^3$, where $\langle ~, ~ \rangle$ denotes the non-pluripolar product (cf. \cite{BEGZ10}). Note that similar expectation is true for the surface case $X=\Bl_{x_0}\PP^2$ as proved in \cite{DMS24}. For the pair of K\"ahler classes $\be= 3[H]-[E]$ and $\al=18[H]-3[E]$ (i.e. $p=2bq=18$ and $q=b=3$) we have that $c_q:= \cot\htheta = \frac{5328}{2863} < 3=q$ and $\xi\approx 3.977$ is the unique root in the interval $(3,6)$ of the fourth degree polynomial 
		\begin{equation*}
			F(s) =  -s^4 + 2887 s^2 -10692 s - 2890,
		\end{equation*}
		and the new dHYM-slope $\zeta$ is given by $$ \zeta = c_{\xi} = \frac{5346-\xi^3 + 3\xi}{2890-3\xi^2} > c_{q}.$$
	\end{rem}

	\begin{rem}
		There is another geometric flow, called the {\em line bundle mean curvature flow} (LBMCF), in the dHYM literature which was introduced and studied in \cite{JY}. Recently Chan-Jacob \cite{ChaJac23} gave some interesting examples of both finite and infinite time singularity formation for LBMCF on $X=\Bl_{x_0}\PP^n$. Since the cotangent flow exists for {\em all time} when the initial data $\vp_0$ satisfies $\theta_{\om}(\chi_{\vp_0})<\pi$ \cite{FYZ} and this condition is {\em sharp} to define the cotangent flow, this rules out the finite time singularity phenomenon unlike the LBMCF. But as our main result and a result in \cite{DMS24} give evidence for infinite time singularity formation, it will be fascinating to find more examples of this types along higher co-dimensional subvarieties as well as on higher dimensional general compact K\"ahler manifolds. We want to explore it in future work.
	\end{rem}
	
	$\textbf{A brief outline of the paper:}$ In $\S$ \ref{sec:preliminaries}, we first recall the Calabi ansatz method for producing K\"ahler metrics on $X=\Bl_{x_0}\PP^n$ in terms of certain real-valued functions and then using this method we will see that the dHYM equation can be reduced to an ordinary differential equation (ODE). Next, we will show that the dHYM stability condition can be interpreted as finitely many basic inequalities involving the constants $p,q$ and $b$. In $\S$ \ref{section:aux family on P n blowup}, we describe the constants $\xi$ and $\zeta$, and the fourth degree polynomial $F$ mentioned in Theorem \ref{thm:singular-dhym-P3 blowup-cotangnet flow_Introduction} in general setting $X=\Bl_{x_0}\PP^n$, and in $\S$ \ref{subsec:examples-dhym-unstab-types}, we provide some examples of ``two kinds of dHYM instability" for $n=3$. In $\S$ \ref{sec:proof-main theorem-P3-blowup}, we mention the important properties to be satisfied by the initial function $\psi_0$ to guarantee the long-time existence and convergence of the cotangent flow, and then we prove the Theorem \ref{thm:singular-dhym-P3 blowup-cotangnet flow_Introduction}.

\vspace*{5mm}
\section{\large Preliminaries}
\label{sec:preliminaries}
	
Let $X=\Bl_{x_0}\PP^n$ be the blow up of the complex projective space $\CC\PP^n$ at a point.

\subsection{Calabi symmetry and reduction}
\label{subsec:Calabi-symm-reduction}
	
	Let $E$ denotes the exceptional divisor, and $H$ the pullback of the hyperplane divisor from $\PP^n$ on $X$. Then (cf. \cite{Laz-bk-v1, BHPV04-book})
	\[ H^{1,1}(X,\RR) = \left\{ p[H] - q[E] ~:~ p, q \in \RR\right\},\]
	where $[~]$ denotes the respective Poincar\'e dual. Moreover, a class $p[H] - q[E]$ is {\em K\"ahler} if and only if $p>q>0$; {\em numerically effective} ({\em nef}) if and only if $p\geq q\geq 0$; and {\em big} if and only if $p>\max(0,q)$. We have the following intersection formulae:
	\begin{equation}\label{eq:intersection formulae P^n blown up at one point}
		[H]^n=1,~~ [E]^n=(-1)^{n-1}, ~~[H]\cdot[E]=0.
	\end{equation}

	\vspace*{1mm}
	We choose any holomorphic coordinates $(z^1,\cdots,z^n)$ on the coordinate patch $X\setminus(H\cup E)\cong\CC^{n}\setminus\{0\}$ and set $\rho :=\ln(\lvert z \rvert^2)$. For any smooth function $u: \RR\to\RR$, $$\om := \ddbar u(\rho)$$ is a smooth real closed $(1,1)$-form $\CC^{n}\setminus\{0\}$ and it is a K\"ahler form if both $u', u''>0$. In order for $\om$ to extend to $X$ in a class $\Omega:=p[H]-q[E]$ (say), $u$ must satisfy the following boundary asymptotics: the functions $U_0, U_\infty : (0,\infty) \rightarrow \RR$ defined by $$U_0(r):= u(\ln r) - q\ln r~~ \text{and} ~~ U_\infty(r):=u(-\ln r) + p \ln r$$ must extend by continuity to smooth functions at $r=0$ (cf. \cite{Cal-extremal}). It implies $$\underset{\rho \to -\infty}{\lim} u'(\rho)= q~~ \text{and} ~~\underset{\rho\to\infty}{\lim} u'(\rho) = p.$$
	Additionally, if $p>q>0$ and both $U'_0(0)>0$, $U'_\infty(0)>0$, the extended form $\om$ is a K\"ahler form on $X$ in the K\"ahler class $\Omega$. Any closed real $(1,1)$-form on $X$ constructed in the above manner is said to have {\em Calabi ansatz} (or {\em symmetry}).

	\vspace*{1mm}
	We fix a (normalized) K\"ahler class $\beta = b[H]-[E]$, where $b>1$, and a real $(1,1)$-cohomology class $\alpha = p[H]-q[E]$, where $p,q\in \RR$. From \eqref{eq:intersection formulae P^n blown up at one point} we have $$(\al+\ii\be)^n = (p+ \ii b)^n - (q+ \ii)^n.$$ Since $\hat\theta\in(0,\pi)$, $\im(p+\ii b)^n > \im(q+\ii)^n$. Choose any K\"ahler form $\om\in\be$ and a real $(1,1)$ form $\chi\in\al$, and assume that both satisfy Calabi symmetry, i.e. $$\om=\ddbar u(\rho)~~ \text{and}~~ \chi=\ddbar v(\rho)$$ for some $u, v \in C^\infty(\RR,\RR)$ satisfying above asymptotic conditions, namely
	\begin{align*}
		\underset{\rho \to -\infty}{\lim} u'(\rho)= 1, ~~ \underset{\rho \to \infty}{\lim} u'(\rho)= b, ~~ \underset{\rho \to -\infty}{\lim} v'(\rho)= q, ~~ \underset{\rho \to \infty}{\lim} v'(\rho)= p.
	\end{align*}
	Restricting to $\CC^n\setminus\{0\}$, the eigenvalues of $\om^{-1}\chi$ are $\frac{v'}{u'}$ with multiplicity $(n-1)$, and $\frac{v''}{u''}$ with multiplicity one. Note that both $u',u''>0$. We set $x:=u'(\rho)\in(1,b)$. Then we let $\psi:[1,b]\to\RR$ be the smooth function defined by
	\begin{equation}
		\label{eq:function psi}
		\psi(u'(\rho)) = v'(\rho),~~\rho\in\RR,
	\end{equation}
	with $\psi(1)=q$ and $\psi(b)=p$. Differentiating with respect to $\rho$ one can see that the above eigenvalues are $\frac{\psi}{x}$ with multiplicity $n-1$ and $\psi'$ with multiplicity one.

	\vspace*{1mm}
	In above setting, the dHYM equation $\eqref{eq:DHYM}$ is reduced to the following exact ODE:
	\begin{equation*}
		\frac{d}{dx}(\text{Re}(\psi+ \ii x)^n -\cot\hat\theta\cdot\text{Im}(\psi+ \ii x)^n)=0.
	\end{equation*}
	Thus, as in \cite{JacSh}, solving $\eqref{eq:DHYM}$ on $X=\Bl_{x_0}\PP^n$ is equivalent to finding a smooth function $\psi:[1,b]\to\RR$ with $\psi(1)=q$, $\psi(b)=p$ which satisfies
	\begin{equation}\label{eq:DHYM ODE solution on P^n blown up at one point}
		\text{Re}(\psi+ \ii x)^n - \cot\hat\theta\cdot \text{Im}(\psi+ \ii x)^n = A_q,
	\end{equation}
	where the constants $c_q$ and $A_q$ are given by
	\begin{equation*}
		\begin{split}
			c_q &:= \cot\hat\theta =\frac{\re(p+ \ii b)^n - \re(q+ \ii)^n}{\im(p+ \ii b)^n - \im(q+ \ii)^n}, \\
			A_q &:= \text{Re}(p + \ii b)^n - c_q \text{Im}(p + \ii b)^n \\ & ~ = \text{Re}(q + \ii)^n - c_q \text{Im}(q + \ii)^n.
		\end{split}
	\end{equation*}
	For the values $p=18$, $q=b=n=3$ we have $c_q=\frac{5328}{2863}$ as mentioned in Remark \ref{rem:explicit values of c_q 4th deg poly and zeta}.

	\vspace*{1mm}
	Now suppose for each $t\geq 0$, $\chi_t:= \chi + \ddbar\vp(t)$ satisfies Calabi ansatz, i.e.  $$\chi_t = \ddbar v(\rho,t)$$ for some $v(\cdot,t)\in C^\infty(\RR,\RR)$. Let $\psi: [1,b]\times [0,\infty) \rightarrow \RR$ be the function defined by
	\begin{equation}\label{function psi t}
		\psi(x,t) := \psi( u'(\rho), t) := v'(\rho, t), ~~\rho\in\RR,
	\end{equation}
	with boundary conditions $\psi(1,t)=q$ and $\psi(b,t)=p$ for any $t\geq 0$. Note that $\psi(\cdot,t)$ is smooth on $(1,b)$. Taking derivatives with respect to $\rho$ and $t$ respectively we obtain
	\begin{equation}\label{eq:cot-flow-reduction-1st-equation}
		\psi'(x,t) u''(\rho) = v''(\rho, t),~~~ \dot\psi(x,t) = \dot v'(\rho, t).
	\end{equation}
	Throughout the paper, we use the following notations $$\dot\psi:=\frac{\d\psi}{\d t}, ~~ \psi':=\frac{\d\psi}{\d x}, ~~ u':=\frac{\d u}{\d\rho}, ~~ u'':=\frac{\d^2 u}{\d \rho^2}.$$ 
	In above formulation, the Lagrangian phase function $\theta_{\om}(\chi_t)$ is given by \[ \theta(x,t):= (n-1)\arccot(\frac{\psi}{x}) + \arccot(\psi').\]
	We first re-write the cotangent flow \eqref{eq:dHYM cotangent flow} as \[\dot v(\rho, t) =  \cot\theta(x,t) - \cot\hat\theta.\]
	Then taking derivative of this equation with respect to $\rho$ and using \eqref{eq:cot-flow-reduction-1st-equation} we see that the cotangent flow \eqref{eq:dHYM cotangent flow} on $X=\Bl_{x_0}\PP^n$ is reduced into the following evolution equation of one space variable:
	\begin{equation}\label{eq:cotangent flow on P^n blowup at a point}
		\dot \psi(x,t) = Q(x)(\cot\theta)' = Q(x)(\csc^2\theta) \big( \frac{\psi''}{1+\psi'^2} + (n-1)\frac{x\psi' -\psi}{x^2 + \psi^2}\big) =: L_2(\psi).
	\end{equation}
	Here $Q:[1,b]\to\RR_{+}$ is defined by $Q(x):=u^{\prime\prime}\circ{u^{\prime}}^{-1}(x)$
	with $Q(1)=Q(b)=0$. Note that $Q$ is a continuous function which is smooth and positive in the interior. The equation \eqref{eq:cotangent flow on P^n blowup at a point} is strictly parabolic away from the boundary points $\{1,b\}$.

	\vspace*{1mm}
	We have the following important observation. 
	
	\begin{lem}\label{lem:supercritical equiv to positivity and monotonicity_blowup of P n}
		For any $\psi\in C^2([1,b],\RR)$ we define 
		\begin{equation}\label{eq:Lag phase func with Calabi ansatz on blow up of CPn at a point}
			\theta_\psi(x) := \arccot (\psi') + (n-1)\arccot(\frac{\psi}{x}). 
		\end{equation}
		The function $\theta_{\psi} < \pi$ if and only if the function $\mathrm{Im}(\psi+\ii x)^n$ is strictly increasing on $[1,b]$. Moreover, $\theta_{\psi} < \pi$ also implies the functions $\mathrm{Im}(\psi+\ii x)^k$ are strictly increasing and positive for all integers $k\in[1, n-1]$, and $\psi$ is positive when $n\geq3$. Furthermore, if $\theta_\psi \equiv \htheta$ for some constant $\htheta\in(0,\pi)$, then for each integer $k\in[1, n-1]$ and all $x\in[1,b]$ we have 
		$$ \mathrm{Re}(\psi + \ii x)^k - (\cot\htheta) \mathrm{Im}(\psi + \ii x)^k >0. $$
	\end{lem}
	
	\begin{proof}
		Observe that the condition $\theta_\psi < \pi$ implies $$0< k \cdot \arccot(\frac{\psi}{x}) < \arccot(\psi') + k \cdot \arccot(\frac{\psi}{x}) < \pi$$ for all integers $k\in[1, n-1]$. As a consequence, $\mathrm{Im}\left(\frac{\psi}{x}+\ii\right)^k>0,$ and so $\im(\psi+\ii x)^k>0$ for all integers $k\in[1,n-1]$. Also note that $\psi>0$ when $n\geq 3$ since if $\psi\leq 0$ at some point, then at that point we must have $\theta_\psi>(n-1)\frac{\pi}{2}\geq \pi$. Using the formula $\cot(kA) = \frac{\mathrm{Re}(\cot A +\ii)^k}{\mathrm{Im}(\cot A +\ii)^k}$, where $k\in\NN$ and $A\in(0,\frac{\pi}{k})$, we get
		\begin{align*}
			\cot(k\cdot\arccot\frac{\psi}{x}) > \cot(\pi-\arccot\psi') \iff  &\mathrm{Re}(\psi+\ii x)^k + \psi' \mathrm{Im}(\psi+\ii x)^k >0 \\
			\iff &\left(\mathrm{Im}(\psi+\ii x)^{k+1}\right)' >0.
		\end{align*}
		This proves the first part. For the second part, suppose $\theta_\psi\equiv \htheta$ for some constant $\htheta\in(0,\pi)$. Then for any integer $k\in[1, n-1]$ we get $$ 0< k \cdot \arccot(\frac{\psi}{x}) < \htheta < \pi,$$ and hence
		$$ \frac{\mathrm{Re}(\psi + \ii x)^k}{\mathrm{Im}(\psi + \ii x)^k} = \cot(k\cdot \arccot \frac{\psi}{x}) > \cot\htheta.$$ 
		We get the desired result since the denominator on the left hand side is positive.
	\end{proof}

\vspace*{3mm}
\subsection{Stability Inequalities}
\label{subsec:stab-ineq}
	
Analogous to \cite[Lemma 1]{JacSh}, we have the following.
	
	\begin{lem}\label{lem:stability for the triple X alpha be on blow up of P n}
		Suppose $X=\Bl_{x_0}\PP^n$, $\al=p[H]-q[E]$ and $\be=b[H]-[E]$ as above with $\htheta\in(0,\pi)$. The triple $(X,\al,\be)$ is dHYM stable (resp. semi-stable) (in the sense of Definition \ref{def:stability of the triple X alpha be}) if and only if for all integers $k\in[1, n-1]$ we have
		\begin{equation}\label{eq:semi-stability-dHYM-blow up of P n}
			\begin{cases}
				\re(p+ \ii b)^k - (\cot\hat\theta) \im(p+ \ii b)^k > 0 ~~(\text{resp.}~\geq), \vspace*{1mm} \\  \re(q+ \ii)^k  - (\cot\hat\theta) \im(q+ \ii)^k >0 ~ ~~(\text{resp.}~\geq).
			\end{cases}
		\end{equation}
	\end{lem}
	
	\begin{proof}
		It is enough to consider the following $m$-dimensional subvarieties $Y= H^{n-m}$ and $Y= (-1)^{n-m-1}E^{n-m}$ where $m\in[1,n-1]$ is an integer. Using \eqref{eq:intersection formulae P^n blown up at one point} we get
		$$ (\al+ \ii\be)^m \cdot H^{n-m} 
		= (p+ \ii b)^m,  ~~ (\al+ \ii\be)^m \cdot (-1)^{n-m-1}E^{n-m}  = (q+ \ii)^m. $$
		Therefore, we get the desired result.
	\end{proof}
	
	In view of the above lemma, the dHYM equation \eqref{eq:dHYM:PDE version} has a smooth solution if and only if stability inequalities \eqref{eq:semi-stability-dHYM-blow up of P n} are satisfied. In particular for $n=2$, the dHYM (resp. semi-) stability of the triple $(X,\al,\be)$ is equivalent to 
	\begin{equation}\label{eq:dim 2 dHYM stability-blow up P 3 at one point}
		p - b c_q >0,~~ q - c_q >0 ~~~(\text{resp.}~ \geq).
	\end{equation}
	On the other hand, for $n\geq 3$ we have the following elementary observation.
	
	\begin{lem}\label{lem:dHYM stab ineqs for CP3 blowup at a point}
		Suppose $n\geq 3$. Then the dHYM stability (resp. semi-stability) of the triple $(\Bl_{x_0}\PP^n, \al,\be)$ implies
		\begin{equation}\label{eq:dim 3 dHYM stability-blow up P 3 at one point}
			p > b\left(c_q + \sqrt{1+c_q^2}\right),~~ q > \left(c_q + \sqrt{1+c_q^2}\right)~~~(\text{resp.}~ \geq).
		\end{equation}
		In particular, both $p$, $q$ must be positive for dHYM semi-stability. Moreover, for $n=3$ the condition \eqref{eq:dim 3 dHYM stability-blow up P 3 at one point} with strict inequality (resp. with $\geq$) is also sufficient for dHYM stability (resp. semi-stability).  
	\end{lem}
	
	\begin{proof}
		Considering $k=1, 2$ in the inequalities \eqref{eq:semi-stability-dHYM-blow up of P n} we get
		\begin{equation*}
			p - b c_q >0,~~ q - c_q >0,~~ p^2-b^2-2bpc_q >0,~~q^2-1-2qc_q >0 ~~~(\text{resp.}~ \geq).
		\end{equation*}
		The last two inequalities implies $$|p-bc_q|>b\sqrt{1+c_q^2}, ~~|q-c_q|>\sqrt{1+c_q^2} ~~~(\text{resp.}~ \geq)$$
		and hence, all observations follow immediately.
	\end{proof}
	
	It is known \cite{DMS24} that the inequality $p>b c_q$ always holds under assumption $\al\cdot\be>0$ for $n=2$. We now show that one stability inequality also comes for free on $\Bl_{x_0}\PP^3$.
	
	\begin{lem}\label{lem:inequality involving p and b holds for blowup of P 3}
		Assume $\mathrm{Im}(\al+\ii\be)^3 >0$, i.e. $3(p^2b - q^2) > b^3 - 1$, and both $p,q>0$. Then $p > b (c_q + \sqrt{1+ c_q^2})$ where 
		\begin{equation}\label{eq:constant c_q for blow up of CP3 at a point}
			c_q = \cot\htheta = \frac{ p^3 - 3pb^2 -q^3 + 3q}{3p^2 b -b^3 -3q^2 +1}.
		\end{equation}
		Furthermore, for the case $0<p\leq q$ we also have $q > (c_q + \sqrt{1+ c_q^2})$, and hence $(\Bl_{x_0}\PP^3,\al,\be)$ is dHYM stable in this case (according to \eqref{eq:dim 3 dHYM stability-blow up P 3 at one point}).
	\end{lem}
	
	\begin{proof}
		Since both $p, q >0$, from the hypothesis we obtain $ 0 < q < p\sqrt{b}$. We compute 
		\begin{align*}
			G(p,q,b) &:=
			(p^2 - b^2 - 2pb c_q) \Big(\mathrm{Im}(p+\ii b)^3 - \mathrm{Im}(q+\ii)^3\Big) \\ &= \mathrm{Im}\Big((p-\ii b)^2((p+\ii b)^3-(q+\ii)^3)\Big) ~\hspace*{1.6cm}(\text{by}~\eqref{eq:first-identity})\\
			&= bp^4 + p^2(2b^3-3q^2+1) + 2pbq(q^2-3) + b^2(b^3 + 3q^2 -1).
		\end{align*}
		
		$\textbf{Case I:}$ $0<p\leq q$ (i.e. the class $\al$ is not K\"ahler). We have $$ 0 < p\leq q < p\sqrt{b}, ~~~ 3(p^2b - q^2) > b^3 - 1.$$ As a consequence we get $3p^2 > b^2 + b +1$, which implies $p>1$ since $b>1$. The above inequalities also imply $c_q <0$. In particular, $ p -b c_q >0$, $q - c_q >0$ and $0< c_q + \sqrt{1+c_q^2} < 1$.  Since $q\geq p >1$, we have $q > c_q + \sqrt{1+c_q^2}$, and $ p > b(c_q + \sqrt{1+c_q^2})$ if we assume $p\geq b$. On the other hand, suppose $p<b$. Then
		\begin{align*}
			G(p,q,b) &= b(p^2 + b^2)^2 + (3q^2 -1)(b^2-p^2) + 2pbq(q^2-3) \\ &> b(p^2 + b^2)^2 - 4pbq ~~~~~~~(\text{since}~q>1~\text{and}~0<p<b) \\ & > (p^2 + b^2)^2 - 4p^2b^2 ~~~~~~~(\text{since}~b>1~\text{and}~0<q<pb)\\ &=(p^2-b^2)^2\geq 0.
		\end{align*}
		Since $p>0$, we must have $p > b(c_q + \sqrt{1+c_q^2})$. As a consequence, the triple $(\Bl_{x_0}\PP^3, \al, \be)$ is always dHYM stable if $q\geq p >0$.
		
		$\textbf{Case II:}$ $p>q>0$ (i.e. the class $\al$ is K\"ahler). We have $$ 0< q < p < p\sqrt{b},~~~3(p^2b-q^2)>b^3-1.$$ If we put $\la:=\frac{p}{q}>1$, then 
		\begin{align*}
			G(p,q,b) = q^4 \la (\la^3 b - 3\la +2 b) + q^2 (2\la^2 b^3 +\la^2 -6\la b +3b^2) + b^2(b^3-1) >0.
		\end{align*}
		Here we used the fact that the functions $f(\la):= \la^3 b -3\la +2b$ and $g(\la):= 2\la^2 b^3 + \la^2 -6\la b +3b^2$ are positive on $(1,\infty)$ since $b>1$. Therefore, it implies that $p > b (c_q + \sqrt{1+c_q^2})$.
	\end{proof}
	
	\begin{rem}
		The inequality  $p > b (c_q + \sqrt{1+c_q^2})$ also holds if $p\in(0,b)$ and $\max(-p\sqrt{b},-\sqrt{3})<q\leq 0$, or if $p\geq b$ and $\max(-p\sqrt{b},-\frac{1}{\sqrt{3}})<q\leq 0$.
	\end{rem}
	
	\begin{rem}\label{rem:only-possibility-of-dhym-unstab}
		For $q\in[\frac{p}{b},p]$ with $p>0$ we shall also show that $q> c_q + \sqrt{1+ c_q^2}$ (see Lemma \ref{lem:stab-interval-p-by-b-to-p}). It implies $(\Bl_{x_0}\PP^3,\al,\be)$ is dHYM stable in this case. Therefore, if $\al$ is a K\"ahler class, the only possibility for occurrence of instability is when $0< q< p/b$. 
	\end{rem}

\vspace*{3mm}
\subsection{Auxiliary family of equations}
\label{section:aux family on P n blowup}
	
	Recall that we fixed real $(1,1)$ class $\al=p[H]-q[E]$. Now, for every $s\in\RR$ we set $$\al_s := p[H]-s[E]=\al + (q-s)[E]. $$  We then define the following subsets of the real line
	\begin{align*}
		\sS_{n} := \left\{ s\in\RR:~ \im(\al_s + \ii\be)^n \neq 0 \right\}, ~~ \sS_{n}^{+} :=  \left\{ s\in\RR:~ \im(\al_s + \ii\be)^n > 0 \right\}.
	\end{align*}
	Analogous to $\S 4.2$ in \cite{DMS24}, we consider the following auxiliary family of equations which are dHYM equations for the triple $(X = \Bl_{x_0}\PP^n, \al_s, \be)$:
	\begin{equation}
		\label{eq:family of dHYM equ on P n blowup}
		\left(\mathrm{Re}(\psi + \ii x)^n\right)' = c_s \left(\mathrm{Im}(\psi + \ii x)^n\right)', ~~\text{and}~~ \psi(1)= s,~ \psi(b) = p,
	\end{equation}
	where $s\in \sS_{n}$ and 
	$$c_s:= \frac{\re(\al_s + \ii\be)^n}{\im(\al_s + \ii\be)^n} = \cot \htheta(X, \al_s, \be). $$
	
	Solutions of the family \eqref{eq:family of dHYM equ on P n blowup} are smooth functions $\tpsi_s:[1,b]\rightarrow\RR$ with above boundary conditions and satisfying
	\begin{equation}\label{eq:family of dHYM solution on P n blowup}
		\mathrm{Re}(\tpsi_s + \ii x)^n -c_s \mathrm{Im}(\tpsi_s + \ii x)^n = A_s,
	\end{equation}
	where
	\begin{equation*}
		\begin{split}
			&c_s = \frac{\mathrm{Re}(p +\ii b)^n - \mathrm{Re}(s +\ii)^n}{\mathrm{Im}(p +\ii b)^n - \mathrm{Im}(s +\ii)^n}, \\
			&A_s = \mathrm{Re}(p +\ii b)^n - c_s \mathrm{Im}(p +\ii b)^n =
			\mathrm{Re}(s +\ii)^n - c_s \mathrm{Im}(s +\ii)^n. 
		\end{split}
	\end{equation*}
	We get such $\tpsi_s$ satisfying \eqref{eq:family of dHYM solution on P n blowup} if and only if $(X, \al_s, \be)$ is dHYM stable \cite{JacSh}, and in view of Lemma \ref{lem:stability for the triple X alpha be on blow up of P n}, if and only if 
	$$ \re(p + \ii b)^k -c_s \im(p+\ii b)^k >0, ~~\text{and}~~ \re(s+\ii)^k -c_s \im(s+\ii)^k >0 $$
	for all integers $k\in[1,n-1]$.

	\vspace*{1mm}
	Using the basic identity 
	$$\mathrm{Im}(z)\mathrm{Re}(w) - \mathrm{Re}(z)\mathrm{Im}(w) = \mathrm{Im}(z\bar w)$$ for any $z,w\in\CC$, we get following identities for any integer $k\in[1,n]$:
	\begin{equation}\label{eq:first-identity}
		\begin{split}
			&\left(\mathrm{Re}(p+\ii b)^k - c_s \mathrm{Im}(p+\ii b)^k \right)\cdot \im(\al_s + \ii\be)^n \\ &=
			\mathrm{Im}\left((p-\ii b)^k\left((p+\ii b)^n -(s+\ii)^n\right)\right), 
		\end{split}
	\end{equation}
	and
	\begin{equation}\label{eq:second-identity}
		\begin{split}
			&\left(\mathrm{Re}(s+\ii)^k - c_s \mathrm{Im}(s+\ii)^k \right)\cdot \im(\al_s + \ii\be)^n \\ &= \mathrm{Im}\left((s-\ii)^k\left((p+\ii b)^n -(s+\ii)^n\right)\right).
		\end{split}
	\end{equation}
	Now	applying these identities we obtain
	\begin{equation}\label{eq:derivative-form-c s}
		\frac{d c_s}{d s} 
		= -\frac{n \Big(\mathrm{Re}(s+\ii)^{n-1} - c_s \mathrm{Im}(s+\ii)^{n-1}\Big)}{\mathrm{Im}(p+\ii b)^n -\mathrm{Im}(s+\ii)^n}.
	\end{equation}
	It implies that $c_s$ is strictly decreasing on $\sS_{n}^{+}$ if and only if $$\re(s+\ii)^{n-1} > c_s \im(s+\ii)^{n-1}.$$ 
	
	\begin{lem}\label{lem:monotonicity-psi-s-any-dim}
		Suppose $s\in \sS^{+}_{n}$ and the triple $(X, \al_s, \be)$ is dHYM stable. Then the sequence of functions $\{\tpsi_s\}$ is strictly increasing.
	\end{lem}
	
	\begin{proof}
		Since $(X, \al_s, \be)$ is dHYM stable, there exist a unique $\tpsi_s \in C^\infty([1,b],\RR)$ satisfying \eqref{eq:family of dHYM solution on P n blowup} with $\tpsi_s(1)=s$, $\tpsi_s(b)=p$. Observe that the family \eqref{eq:family of dHYM equ on P n blowup} can also be written as $$\theta_{\tpsi_s} = \arccot(c_s)\in(0,\pi).$$ Therefore, applying Lemma \ref{lem:supercritical equiv to positivity and monotonicity_blowup of P n} the functions $\mathrm{Im}(\tpsi_s + \ii x)^k$ are strictly increasing for all integers $k\in[1,n]$, and positive for $k<n$ on $[1,b]$. Moreover, $\tpsi_s >0$ when $n\geq 3$, and we have $ \mathrm{Re}(\tpsi_s + \ii x)^k > c_s \mathrm{Im}(\tpsi_s + \ii x)^k$ for $k<n$. Differentiating \eqref{eq:family of dHYM solution on P n blowup} with respect to $s$ we obtain
		\begin{align*}
			& n \left(\mathrm{Re}(\tpsi_s +\ii x)^{n-1} -c_s \mathrm{Im}(\tpsi_s +\ii x)^{n-1}\right) \frac{d \tpsi_s}{d s} \\ &= - \Big(\mathrm{Im}(p+\ii b)^n - \mathrm{Im}(\tpsi_s +\ii x)^n\Big) \frac{d c_s}{d s}.
		\end{align*}
		Therefore, we get the desired result because both the terms in brackets are positive and $\frac{dc_s}{d s}<0$ by the hypothesis. 
	\end{proof}
	
	We let 
	$$F(s):= \mathrm{Im}\Big((s-\ii)^{n-1}\big((p+\ii b)^n - (s+\ii)^n\big)\Big),$$ 
	which is a polynomial of degree $2n-2$. Using \eqref{eq:second-identity} (for $k=n-1$) we also see that
	\begin{align}
		F(s) = \left(\mathrm{Re}(s+\ii)^{n-1} -c_s \mathrm{Im}(s+\ii)^{n-1}\right)\im(\al_s + \ii\be)^n
	\end{align}
	for all $s\in\sS_{n}$. One can check that $$ F\left(\frac{p}{b}\right) = \frac{(b^n-1)(p^2+b^2)^{n-1}}{b^{2n-2}} >0.$$
	Differentiating the polynomial $F$ we obtain
	\begin{align*}
		&\frac{1}{n-1} F'(s) \\ &=  \mathrm{Im}\left((s-\ii)^{n-2}\left((p+\ii b)^n - (s+\ii)^n\right)\right) \\ &= \left(\re(s+\ii)^{n-2} - c_s \im(s+\ii)^{n-2}\right) \im(\al_s + \ii \be)^n. ~~~~(\text{by}~\eqref{eq:second-identity})
	\end{align*}
	
	Similar to Lemma $4.4$ in \cite{DMS24} for $n=2$ case, we have the following result.
	
	\begin{lem}\label{lem:bigness-blowup-P 3-family s}
		Suppose $n=3$ and $p>0$. Then for $s\in \sS_{3}^{+}\cap(0,\infty)$ we have $p-bc_s >0$. In particular,  the class $\al_s - c_s \be$ is big for all 
		$s \in \sS^{+}_{3}\cap \left(0, \frac{p}{b}\right)$.
	\end{lem}
	
	\begin{proof}
		Since any $s\in\sS^{+}_{3}\cap(0,\infty)$ satisfies the assumptions on $q$ in Lemma \ref{lem:inequality involving p and b holds for blowup of P 3}, we obtain $p > b(c_s + \sqrt{c_s^2 +1})$, and hence $c_s< \frac{p}{b}$. Now for any  $0<s\leq\frac{p}{b}$ we have $\frac{p-s}{b-1}\geq\frac{p}{b}$. Since $ \al_s - c_s\be = (p-b c_s)[H] - (s- c_s)[E]$, it is a big class if $p-bc_s > \max( s-c_s, 0)$, and this holds for any $s \in \sS^{+}_{3}\cap \left(0, \frac{p}{b}\right)$.
	\end{proof}
	
	As mentioned in Remark \ref{rem:only-possibility-of-dhym-unstab} we now show that $(\Bl_{x_0}\PP^3,\al,\be)$ is dHYM stable for $q\in[p/b,p]$ (see Lemma \ref{lem:stab-interval-p-by-b-to-p}). For $n=3$ the polynomial $F$ defined above is given by
	\begin{equation}\label{eq:poly-F-in-dim-3-explicit-form}
		\begin{split}
			F(s) &= -s^4 + s^2(3p^2b - b^3 -2) -2sp(p^2 - 3b^2) - (3p^2b -b^3 + 1) \\ &= (s^2 -1 -2s c_s) \im(\al_s + \ii \be)^3.
		\end{split}
	\end{equation}
	Note that putting $p=18$, $q=b=n=3$ we get the explicit expression of $F$ as mentioned in Remark \ref{rem:explicit values of c_q 4th deg poly and zeta}. Differentiating we obtain
	\begin{equation}\label{eq:derivatives-poly-F-in-dim-3-explicit-form}
		\begin{split}
			&F'(s) = 2(-2s^3 + (3p^2b -b^3 -2)s - (p^3 - 3pb^2)) = 2(s - c_s) \im(\al_s + \ii \be)^3,\\
			&F''(s) = 2(-6s^2+(3p^2b -b^3 -2)). 
		\end{split}
	\end{equation}
	We already showed that $F(p/b)>0$ for any dimension $n$. Now, a straightforward computation also implies $F(\sqrt{\frac{3p^2b -b^3 +1}{3}}) > 0$ for $n=3$ and $\sqrt{\frac{b^3-1}{3b}} < p < \frac{b}{\sqrt{3}}$.
	
	\begin{lem}\label{lem:stab-interval-p-by-b-to-p}
		Suppose $n=3$. The polynomial $F$ is positive (and hence $s > c_s + \sqrt{c_s^2 +1}$) for all $s\in[\frac{p}{b}, p]$ and $3(p^2b -s^2) > b^3 -1$ (with $p>0$).
	\end{lem}
	
	\begin{proof}
		Observe that since $s\geq \frac{p}{b}$ we must have $p > \frac{b}{\sqrt{3}}$. Then $$ F(p) = (b-1)(3p^4 + p^2(-b^3 + 5b +2) + b^2 + b +1) > 0.$$
		Now, if $p\leq \sqrt{\frac{b^3+2}{3b}}$, the polynomial $F$ is concave on $\RR_{+}$, and we get that $F$ is positive on $[\frac{p}{b}, p]$. While for $p> \sqrt{\frac{b^3+2}{3b}}$ and $s\in[p/b, p]$ we get $$ \frac{1}{2} F'(s) \geq -2p^3 + (3p^2b -b^3 -1)\frac{p}{b} -p^3 + 3pb^2 = \frac{2p(b^3 -1)}{b} >0,$$ and hence again $F$ is positive on $[\frac{p}{b}, p]$. As a consequence, we get that $s > c_s + \sqrt{c_s^2 +1}$ for all $s\in[\frac{p}{b}, p]$ and $3(p^2b -s^2) > b^3 -1$. 
	\end{proof}
	
	Motivated by Lemma $4.6$ in \cite{DMS24},  we define the following constant for $X=\Bl_{x_0}\PP^n$
	$$\xi := \inf \left\{ s\geq q:~ s\in\sS^{+}_{n}, ~~ \re(s+\ii)^{k} \geq c_s \im(s+\ii)^{k}~~\forall~k = 1, \cdots, n-1\right\}. $$
	For $n=2$, we have $\xi= bp - \sqrt{(b^2-1)(p^2 +1)}$, and there exists a unique $\tpsi_{\xi}$ solving \eqref{eq:family of dHYM solution on P n blowup} (cf. \cite{DMS24}). While for $n=3$ the following hold:
	\begin{equation*}
		q \leq \xi < \min\big(\frac{p}{b}, \sqrt{\frac{3p^2b-b^3+1}{3}}\big) =: p_{\star}~(\text{say}), ~~ \xi > c_{\xi}, ~~\text{and}~~
		\xi^2 - 1 - 2\xi c_\xi = 0.
	\end{equation*}
	Observe that $\xi=q$ if and only if $ q\geq c_q + \sqrt{c_q^2 + 1}$, i.e., if and only if $(X,\al,\be)$ is dHYM semi-stable. On the other hand, we get $\xi>q$ in the case $0< q < c_q + \sqrt{c_q^2 + 1}$. Since we assumed $q>0$ for the unstable case, it is easy to see that $s^2 -1 -2s c_s < 0$ for all $s\in(q,\xi)$. We now define the following constant for $n=3$:
	\begin{equation}\label{eq:xi prime defn for blowup pf CP3 at a point}
		\xi' : = \sup \left\{ s \in (\xi, p_{\star}) : ~ t^2 -1 - 2t c_t > 0 ~\text{for all}~ t\in (\xi, s) \right\}.
	\end{equation}
	Then $s^2 -1 -2s c_s > 0$ for all $s\in (\xi, \xi')$, and in particular, $s > c_s$ since $\xi>0$. It implies that $\xi$ is the {\em unique solution} of the polynomial $F$ in the interval $(q, \xi')$. Moreover, $F$ is positive on $(\xi, \xi')$ and negative on $(q, \xi)$. In particular, by \eqref{eq:derivative-form-c s} we see that $c_s$ is strictly increasing (resp. decreasing) on $s\in(q, \xi)$ (resp. $s\in(\xi, \xi')$), and attains maximum value at $s=\xi$. Note that $p_{\star}=p/b$ if and only if $p\geq \frac{b}{\sqrt{3}}$.
	
	\begin{rem}
		For $n=2$ the corresponding $\xi'$ is defined as the supremum of all $s\in(\xi,bp)$ such that $t>c_t$ for all $t\in(\xi,s)$, and then $\xi'=bp$ as in \cite{DMS24}. But for $n\geq 4$ it is {\em not} clear how to define the constant $\xi'$ so that $\xi$ will be the {\em unique} solution of the polynomial $F$ (for general $n$) in the interval $(q, \xi')$.
	\end{rem}
	
	In terms of the constant $c_s$ given above (with $n=3$), the new dHYM slope $\zeta$ defined in \eqref{eq:new dHYM slope zeta for CP3 blowup at a point} is given by
	\begin{align*}
		\zeta = \sup \left\{ c_s : ~ s\in \sS_{3}^{+}\cap(0,\infty) \right\}.
	\end{align*}
	Now from the derivative formula \eqref{eq:derivative-form-c s} of $c_s$ with respect to the variable $s$ and from the definition of the constants $\xi$ and $\xi'$ given above, we easily see that $c_s$ is strictly increasing on the interval $(q,\xi)$ and strictly decreasing on the interval $(\xi, \xi')$. In particular, it attains maximum at $s=\xi$. Therefore, we must have $\zeta = c_{\xi}$.

	\vspace*{1mm}	
	Similar to \cite{DMS24} for $n=2$, we have a unique branch $\tpsi_{\xi}$ solving \eqref{eq:DHYM ODE solution on P^n blown up at one point} with $\tpsi_{\xi}(1)=\xi$, $\tpsi_{\xi}(b)=p$ for $n=3$. Indeed, taking $s=\xi$ and $n=3$ in \eqref{eq:DHYM ODE solution on P^n blown up at one point} we get
	\begin{equation}
		\tpsi_{\xi}^{3} - 3x^2\tpsi_{\xi} - c_{\xi}(3x\tpsi_{\xi}^{2} - x^3) = A_{\xi}.
	\end{equation}
	In fact, the constants $c_{\xi}$ and $A_{\xi}$ are given by $c_{\xi}=\frac{\xi^2 -1}{2\xi}$ and $A_{\xi} = -\frac{(\xi^2 +1)^2}{2\xi}$. Now if we consider $\Psi_{\xi}:=\tpsi_{\xi} - c_{\xi}x$, then $\Psi_\xi$ satisfies the following depressed cubic equation
	\begin{equation}\label{eq:depressed-cubic-blowup-P3}
		\Psi_{\xi}^{3} + R_{1}\Psi_{\xi} + R_2 = 0,
	\end{equation}
	where $R_1 = -3(c_{\xi}^2 +1)x^2$ and $R_2= -2c_{\xi}(c_{\xi}^2 +1)x^3 - A_{\xi}$. Then the discriminant of \eqref{eq:depressed-cubic-blowup-P3} is $\Delta:=-(4R_{1}^{3} + 27 R_{2}^{2})$, i.e.
	$$ \Delta=\Delta(x) = 108(c_{\xi}^2 +1)^{2}(x^6 + 2\xi c_{\xi} x^3 -\xi^2)$$ which is positive on $(1,b]$ and zero at $x=1$. Thus, we get a unique such branch $\tpsi_{\xi}$.

\vspace*{3mm}
\subsection{Examples of two kinds of dHYM instability}
\label{subsec:examples-dhym-unstab-types}
	
Recall $(X=\Bl_{x_0}\PP^3,\al,\be)$ is dHYM stable if and only if $q\geq c_q + \sqrt{1+ c_q^2}$. In particular, if $q\leq 0$, it is dHYM unstable. On the other hand, we proved in Lemma \ref{lem:inequality involving p and b holds for blowup of P 3} and Lemma \ref{lem:stab-interval-p-by-b-to-p} that it is dHYM stable if $q\geq\frac{p}{b}>0$.

	\vspace*{1mm}
	We now give an example of $c_q < q \leq c_q + \sqrt{1+ c_q^2}$, where $c_q$ could be positive or zero or negative, and we say this is {\em first kind of dHYM instability} where instability occurs along co-dimension one subvarieties. Next we give an example of $0< q \leq c_q$ which we say {\em second kind of dHYM instability} where instability occurs along co-dimension two subvarieties.

	\vspace*{1mm}
	For simplicity, if we put $\la:=p/s$, then from \eqref{eq:poly-F-in-dim-3-explicit-form} and \eqref{eq:derivatives-poly-F-in-dim-3-explicit-form} we get
	\begin{align*}
		F(s) &= s^4 (-2\la^3 + 3\la^2 b -1) + s^2 (-3\la^2b + 6\la b^2 -b^3 -2) + (b^3-1) \\
		F'(s) &= 2 s^3(-\la^3 + 3\la^2 b - 2) + 2s(3\la b^2 -b^3 -2).
	\end{align*}
	Now considering $\la=3b$ we get
	$$ 2(s - c_s)\big(3(9b^3 -1)s^2 - (b^3 -1)\big) =  F'(s) = 4s (4b^3 -1 -s^2). $$ It implies $0<s \leq c_s$ if $s\geq \sqrt{4b^3 -1}$. Therefore, replacing $s$ by $q$, we get second kind of dHYM instability $0< q \leq c_q$ when $p=3bq$ and $q \geq \sqrt{4b^3 -1}$.

	\vspace*{1mm}
	Since the polynomial $P_0(\la):= -\la^3 + 3\la^2 b -2$ is positive on $[a', a'']$, where $a'\in(0,1)$ and $a''\in(2b,3b)$ are only positive zeros of $P_0$, and $a' < \frac{3b^3 + 2}{3b^2}$ because $ P_0 \left(\frac{3b^3 + 2}{3b^2}\right)>0$, we must have $s > c_s$ if $\frac{3b^3 + 2}{3b^2} \leq \frac{p}{s} \leq a''$ (and both $p,s$ are positive).

	\vspace*{1mm}
	Now consider $\la= 2b$ (i.e. $p=2bs$). Since $\frac{p}{s}=2b \in \left(\frac{3b^3 + 2}{3b^2}, a''\right)$, we have $s > c_s$. For this choice $$ F(s) = -(4b^3 + 1)s^4 - (b^3 + 2)s^2 + b^3 -1. $$ We can see that $F(s)<0$ for all $s > b_{\star}$, where
	$$ b_\star := \left(\frac{\sqrt{b^3(17b^3 -1)} -b^3 -2}{2(4b^3 +1)}\right)^{1/2}, $$ 
	and hence we get that $0< s < c_s + \sqrt{c_s^2 + 1}$ for the choice $p=2bs$ and $s > b_\star$. Moreover, $c_s$ is positive when $s > \left(\frac{3(2b^3-1)}{8b^3 -1}\right)^{1/2}=: b^\star$, equals to zero when $s = b^\star$, and negative when $b_\star < s < b^\star$. This can be observed from the following inequalities $\sqrt{\frac{b^3-1}{3(4b^3-1)}} < b_{\star} < b^\star < 1$. Therefore, once again replacing $s$ by $q$, we get first kind of dHYM instability $c_q < q < c_q + \sqrt{c_q^2 +1}$ when $p = 2bq$ with $q> b_\star$. Note that the explicit example mentioned in Remark \ref{rem:explicit values of c_q 4th deg poly and zeta} is of this kind.

\vspace*{5mm}
\section{\large Proof of the main theorem}
\label{sec:proof-main theorem-P3-blowup}
	
In this section we prove Theorem \ref{thm:singular-dhym-P3 blowup-cotangnet flow_Introduction}. As we will see below, the key point is to find an appropriate initial Calabi potential $\psi_0$ so that the dHYM cotangent flow exists for all time and converges to some limit. More precisely, as in \cite{DMS24} for $n=2$ case, we assume that there exists an initial $\psi_0\in C^\infty([1,b],\RR)$ with $\psi_0(1)=q$, $\psi_0(b)=p$ satisfying $\theta_{\psi_0} <\pi$ and satisfies the condition \eqref{eq:comparison and monotonicity for initial function_Introduction}, namely
	\begin{equation}\tag{\textbf{H1}}
		\label{eq:comparison and monotonicity for initial function}
		\psi_0 \leq \tpsi_\xi ~~\text{and}~~ (\cot \theta_{\psi_0})' > 0.
	\end{equation}
	From the flow equation \eqref{eq:cotangent flow on P^n blowup at a point} it follows that the second condition of \eqref{eq:comparison and monotonicity for initial function} implies
	\begin{equation}\label{eq:t-derivative-at-zero-nonnegative-assumption}
		\dot\psi(x,0)\geq 0.
	\end{equation}
	Moreover, from Lemma \ref{lem:supercritical equiv to positivity and monotonicity_blowup of P n} we see that the condition $\theta_{\psi_0}<\pi$ implies $\theta_\om(\chi_0)<\pi$, where $\chi_0$ is the closed real $(1,1)$-form corresponding to the Calabi potential $\psi_0$. In $\S$ \ref{subsec:initial-function-a-choice}, we will define an initial function $\psi_0$ which satisfies all but the first condition of \eqref{eq:comparison and monotonicity for initial function} for arbitrary $b,p,q$ with $b>1$ and $0<q<p_{\star}$ (see $\S$ \ref{section:aux family on P n blowup} for the constant $p_{\star}$). Moreover, for the particular choice $p=2bq=18$ and $q=b=3$ the first condition of \eqref{eq:comparison and monotonicity for initial function} also holds for this initial function (see Lemma \ref{lem:initial-func-less-than-limit}).

	\vspace*{1mm}
	Assuming the existence of such an initial function $\psi_0$ satisfying $\theta_{\psi_0}<\pi$ and \eqref{eq:comparison and monotonicity for initial function}, we will prove the Theorem \ref{thm:singular-dhym-P3 blowup-cotangnet flow_Introduction}.

\vspace*{3mm}
\subsection{Comparison and monotonicity along the flow}\label{subsec:dhym-key-lem}

We now prove a comparison and monotonicity result along the cotangent flow under the the assumption that there exists an initial function $\psi_0$ satisfying \eqref{eq:comparison and monotonicity for initial function} for any dimension $n$. The arguments are basically appropriate modification of the Lemmas $4.9$ and $4.10$ in \cite{DMS24}.
	
	\begin{lem}\label{lem:comparison-to-limit-gen dim} 
		Suppose the first assumption of \eqref{eq:comparison and monotonicity for initial function} holds and assume that there exists a unique branch $\tpsi_{\xi}$ of the level set \eqref{eq:DHYM ODE solution on P^n blown up at one point} with $\tpsi_{\xi}(1)=\xi$ and $\tpsi_{\xi}(b)=p$. (Note that for $n=2$ and $n=3$ we have such a unique $\tpsi_{\xi}$.)
		For all $t\in [0,\infty)$ and $x\in [1,b]$, we have $\psi(x,t) \leq \tilde\psi_\xi(x)$ along the cotangent flow.
	\end{lem}
	
	\begin{proof}
		It is enough to show that $\psi(x,t) \leq \tilde\psi_s(x)$ for all $s\in(\xi, \xi')$. We let 
		$$H_s(x,t) := \psi(x,t) - \tilde\psi_s(x).$$ 
		Then $$\dot H_s =\dot\psi, ~ H'_s = \psi' - \tpsi'_s,~ H''_s = \psi'' - \tpsi''_s.$$ 
		But then from the evolution equation \eqref{eq:cotangent flow on P^n blowup at a point} we obtain the following evolution equation for $H_s$ on $(1, b]\times [0, \infty)$:
		\begin{align*}
			\dot H_s = \frac{Q\csc^2(\theta)}{(x^2+\psi^2)(1+ \psi'^2)}
			\Big( &(x^2+\psi^2)H''_{s} + (n-1)(1+ \psi'^2)(xH'_{s} - H_s) \\
			&+ (H_{s}+2\tilde\psi_s)H_{s}\tilde\psi''_s  + (n-1)H'_{s}(H'_s + 2\tilde\psi'_s)(x\tilde\psi'_s -\tilde\psi_s)\Big),
		\end{align*}
		where we used the fact that $\tilde\psi_s$ satisfies the following second order ODE:
		\begin{equation}
			\label{equ3-gen dim}
			(x^2+\tilde\psi_{s}^2)\tilde\psi''_s + (n-1)(1+ \tpsi'^2_s)(x\tilde\psi'_s -\tilde\psi_s)=0.
		\end{equation}
		At the spatial and time boundaries we observe that, 
		$$ H_s(x, 0) < 0, ~\forall x\in (1, b),~H_s(1,t)<0,~\forall t\in [0,\infty),~H_s(b,t) = 0,~\forall t\in [0,\infty).$$
		Arguing by contradiction, suppose there exists a $T>0$ such that $$H^*:=\underset{[1,b]\times[0,T]}{\sup} H_s(x,t) >0.$$ Note that $$a(x,t):= \frac{Q \csc^2 \theta}{(x^2 +\psi^2)(1+\psi'^2)}$$ is uniformly bounded above as the angle $\theta(x,t)$ lies in a compact set of $(0, \pi)$ and $Q=u''\circ u'^{-1}$. We also have that $||\tilde\psi_s||_{C^{2}[1,b]}$ is finite for $s>\xi$ (Note that this fails at $s = \xi$). Define a constant $A=A(T)>0$ by
		$$ A := 1+ \lVert \tilde\psi_s\rVert_{C^2[1,b]} + \underset{[1,b]\times[0,T]}{\sup} a(x,t).$$ We let $B= 1+A^2H^* + 2A^3$ and consider the function $F(x,t):= e^{-Bt}H_s(x,t)$. First we see that
		$$ F^* := F(x^*,t^*) = \underset{[1,b]\times[0,T]}{\sup} F(x,t) >0$$ for some $(x^*,t^*)\in[1,b]\times[0,T]$. Next, the spatial and time boundary conditions implies that $x^*\in(1,b)$ and $t^*\neq 0$. It implies that $$\dot F(x^*,t^*)\geq0,~~F'(x^*,t^*)=0,~~F''(x^*,t^*) \leq 0.$$ Then by the maximum principle at the point $(x^*,t^*)$ we get
		\begin{align*}
			0 &\leq \dot F = -BF^* + e^{-B t^*} \dot H_s \\
			&= -BF^*+ a(x^*,t^*)\Big((x^2+\psi^2)F'' + (n-1)(1+ \psi'^2)(xF'-F)
			+ F^*(H_s+2\tilde\psi_s)\tilde\psi''_s  \\ 
			&\hspace*{3.5cm} + (n-1)(H'_s + 2\tilde\psi'_s )(x\tilde\psi'_s -\tilde\psi_s)F'(x^*, t^*)\Big)\\
			&\leq -BF^* + a(x^*,t^*)\Big(H(x^*, t^*)+2\tilde\psi_s(x^*)\Big)\tilde\psi''_s(x^*)F^* \\
			& \leq -BF^* +(A^2H^* + 2A^3)F^* = -F^*,
		\end{align*}
		which is a contradiction. Hence, we must have $$\underset{[1,b]\times[0,\infty)}{\sup} H_s(x,t)\leq 0,$$ and so we get the desired result.
	\end{proof}
	
	\begin{lem}\label{lem:dhym-calabi-monotonicity-gen dim} 
		Suppose the second assumption of \eqref{eq:comparison and monotonicity for initial function} holds. For all $t\in [0,\infty)$ and all $x\in [1,b]$, we have $\dot\psi(x,t) \geq 0$ along the cotangent flow.
	\end{lem}
	
	\begin{proof}
		At the boundary points, since $Q(1) = Q(b) = 0$, we have $\dot\psi(1,t) = \dot\psi(b,t) =0$ for all $t$. Taking one time derivative of the evolution equation \eqref{eq:cotangent flow on P^n blowup at a point} we see that $\dot\psi$ satisfies the following equation: $$\ddot\psi = L(\dot\psi) + a\dot\psi,$$ where $L(f) := {\bf A}(x,t)f'' + {\bf B}(x,t)f'$,
		$$\begin{cases}\mathbf{A}(x,t) &:= \frac{Q(x)\csc^2(\theta)}{1+\psi'^2},\\
			\mathbf{B}(x,t) &:= Q(x)\csc^2(\theta) \Big(-2\frac{\psi'\psi''}{(1 + \psi'^2)^2} +(n-1)\frac{x}{x^2 + \psi^2}\Big) + \frac{\cot\theta}{1 + \psi'^2}\dot\psi,\\
			a(x,t) &:= Q(x)\csc^2(\theta)(n-1) \Big(-\frac{1}{x^2 + \psi^2} -2\frac{\psi( x\psi' -\psi)}{(x^2 + \psi^2)^2}\Big) + (n-1)\frac{x \cot\theta}{x^2 + \psi^2} \dot\psi. 
		\end{cases}$$
		Arguing by contradiction, suppose there exists a $T>0$ such that $$\inf_{[1,b]\times[0,T]}\dot\psi(x,t)<0.$$ Let $C = C(T)>0$ such that $$C = 1+\sup_{[1,b]\times[0,T]}|a(x,t)|.$$ We now let $F(x,t) = e^{-Ct}\dot\psi(x,t),$ and let $(x^*,t^*) \in [1,b]\times[0,T]$ such that $$F^*:= F(x^*,t^*) = \inf_{[1,b]\times[0,T]} F(x,t) <0.$$ Then by the boundary conditions, $x^*\in (1,b)$ and $t^*\neq 0$. It implies that $$\dot F(x^*,t^*)\leq0,~~F'(x^*,t^*)=0,~~F''(x^*,t^*) \geq 0.$$  By the maximum principle at the point $(x^*,t^*)$ we get 
		\begin{align*}
			0 \geq \dot F(x^*,t^*) = -C F^* + L(F)(x^*,t^*) + a(x^*,t^*)F^* \geq (a(x^*,t^*)-C)F^* \geq 0
		\end{align*}
		because $a(x^*,t^*) - C\leq -1$ and $F^*<0$, and so we have a contradiction. Therefore, $$\inf_{[1,b]\times[0,\infty)}\dot\psi(x,t)\geq 0,$$
		and so we get the desired result.
	\end{proof}

\vspace*{3mm}
\subsection{A choice of initial function}
\label{subsec:initial-function-a-choice}
	
We choose an initial function $\psi_0:[1,b] \to \RR$ defined as 
	\begin{equation}\label{eq:initial function-dim 3}
		\psi_0(x):= \frac{1}{\sqrt{3}}\sqrt{\mu x^2 + \frac{\la}{x}},
	\end{equation}
	where the constants $\mu$ and $\la$ are given by
	$$\mu = \frac{3(p^2b - q^2)}{b^3-1},~~ \la = 3q^2 - \mu = \frac{3b(q^2b^2 -p^2)}{b^3-1}. $$ 
	Note that $\mu-1>0$  because we have $3q^2 < 3p^2b -b^3 + 1$, and the function $\psi_0$ is well-defined since $\mu x^3 + \la \geq \mu+\la = 3q^2$. Also note $\la<0$ if $q<\frac{p}{b}$. It is easy to see that $\psi_0(1)=q$, $\psi_0(b)=p$. Since $3\psi_0^2 x - x^3 = (\mu-1)x^3 +\la$, we also have $\theta_{\psi_0}<\pi$ (by Lemma \ref{lem:supercritical equiv to positivity and monotonicity_blowup of P n}). Moreover, by straightforward computation we get
	\begin{align*}
		(\cot\theta_{\psi_0})' &= \frac{\csc^2(\theta_{\psi_0})}{(1 + {\psi'_0}^2)(x^2+\psi_0^2)} \left((x^2+\psi_0^2)\psi''_0 + 2(x\psi_0 -\psi_0)(1 + {\psi'_0}^2)\right) \\ &= \frac{9\la^2 (\mu-1)}{4\sqrt{3} x^2}\left(\mu x^2 + \frac{\la}{x}\right)^{-\frac{3}{2}} > 0,
	\end{align*}
	and in particular, from the flow equation \eqref{eq:cotangent flow on P^n blowup at a point} we get that $\dot\psi(x,0)\geq 0$. The remaining thing is to show that the initial function $\psi_0$ given by \eqref{eq:initial function-dim 3} also satisfies $\psi_0 \leq \tpsi_\xi$ on $[1,b]$. For that, we consider the following function for each $q\leq s < p_{\star}$: 
	$$\psi^{(s)}(x):= \frac{1}{\sqrt{3}}\sqrt{\mu^{(s)} x^2 + \frac{\la^{(s)}}{x}},$$ 
	where the constants $\mu^{(s)}$ and $\la^{(s)}$ are given by
	$$\mu^{(s)} = \frac{3(p^2b - s^2)}{b^3-1},~~ \la^{(s)} = 3s^2 - \mu^{(s)}. $$ 
	Note that at $s=q$ the function $\psi^{(s)}$ is exactly the initial function $\psi_0$. We first observe that $\psi^{(s)}$ is increasing in $s$, and in particular, $\psi_0 \leq \psi^{(\xi)}$. Indeed, 
	\begin{align*}
		\mu^{(s)} x^2 + \frac{\la^{(s)}}{x} &= \mu^{(s)}\frac{x^3-1}{x} + \frac{3s^2}{x} \\
		&= \frac{3p^2b(x^3-1)}{(b^3-1)x} + \frac{3s^2(b^3-x^3)}{x(b^3-1)},
	\end{align*}
	and the right hand side is increasing in the variable $s$.

	\vspace*{1mm}
	We next show that $\psi^{(\xi)}\leq \tpsi_{\xi}$ on $[1,b]$ for the choice $p=2bq$ with $b=q=3$. Note that $\psi^{(\xi)}(1)=\tpsi_\xi(1)=\xi$ and  $\psi^{(\xi)}(b)=\tpsi_\xi(b)=p$. Recall that the functions $\tpsi_\xi$ and $\psi^{(\xi)}$ satisfies the following cubic and quadratic polynomials respectively in the variable $y$ with coefficients depending on $x$:
	\begin{align*}
		&y^3 - 3 c_{\xi}xy^2 -3x^{2}y + c_{\xi}x^3 - A_{\xi} = 0, \\
		& 3xy^2 -(\mu^{(\xi)}x^3 + \la^{(\xi)}) = 0.
	\end{align*}
	By straightforward computation, the {\em resultant} of the above polynomials is the following:
	\begin{align*}
		\sR(x) &= (\mu^{(\xi)}x^3 + \la^{(\xi)})^{2}\left((27c_{\xi}^2 + 18 - \mu^{(\xi)})x^3 - \la^{(\xi)}\right) + 27 x^3 (c_{\xi}x^3 - A_{\xi})^2 \\
		&\hspace*{2cm}  -27x^3(\mu^{(\xi)}x^3 + \la^{(\xi)})\left((2c_{\xi}^2 + 3)x^3 - 2c_{\xi}A_{\xi}\right).
	\end{align*}
	The constants $\mu^{(\xi)}$, $\la^{(\xi)}$, $c_{\xi}$ and $A_{\xi}$ have the following expression in terms of $\xi$:
	\begin{equation}\label{eq:values-const-in-terms-of-xi}
		\begin{split}
			&\mu^{(\xi)} = \frac{3(p^2b -\xi^2)}{b^3-1}, ~~ \la^{(\xi)} = 3\xi^2 - \mu^{(\xi)} = \frac{3b(\xi^2 b^2 - p^2)}{b^3 -1}, ~~ c_{\xi} = \frac{\xi^2 -1}{2\xi}, \\ & A_{\xi} = \xi^3 - 3\xi -c_{\xi}(3\xi^2 -1) = -\frac{(\xi^2 +1)^2}{2\xi}.
		\end{split}
	\end{equation}
	Note that $\mu^{(\xi)}>1$, $\la^{(\xi)}<0$ and $A_{\xi}<0$. Using $\mu^{(\xi)} + \la^{(\xi)}=3\xi^2$ one can check that $\sR(1)=0$. Moreover, using $\mu^{(\xi)} b^3 + \la^{(\xi)}=3p^2b$ we also have $\sR(b)=0$ since $\xi$ is a root of the polynomial $F$, i.e. it satisfies the following:
	$$ (\xi^2+1)^2 = (3p^2b -b^3)(\xi^2 -1) - 2\xi p(p^2 -3b^2).$$

	Considering $\tau:=x^3$ in the above polynomial $\sR$ we obtain the following cubic polynomial in the variable $\tau$:
	\begin{align*}
		\sR_{0}(t) &= (\mu^{(\xi)} \tau +\la^{(\xi)})^{2}\left((27c_{\xi}^2 + 18 - \mu^{(\xi)})\tau - \la^{(\xi)}\right) + 27\tau(c_{\xi}\tau - A_{\xi})^2 \\ & ~~~~~~  -27\tau(\mu^{(\xi)} \tau + \la^{(\xi)})\left((2c_{\xi}^2 + 3)\tau - 2c_{\xi}A_{\xi}\right).
	\end{align*}
	We just have seen that two roots of the polynomial $\sR_{0}$ are $\tau=1$ and $\tau=b^3$. By straightforward factorization, the third real root of the polynomial $\sR_{0}$ is given by 
	\begin{equation}\label{eq:third-real-root-expression}
		\tau = -b^3 -1 - \frac{54c_{\xi}(\mu^{(\xi)}-1)(\la^{(\xi)} c_{\xi} + A_{\xi}) - 3\la^{(\xi)} (\mu^{(\xi)}-3)(\mu^{(\xi)}-9)}{27c_{\xi}^2(\mu^{(\xi)}-1)^2 -\mu^{(\xi)}(\mu^{(\xi)}-9)^2}.
	\end{equation}
	Now using \eqref{eq:values-const-in-terms-of-xi}, for the special values $q=b=3$ and $p=2bq=18$, this third root is negative. It implies that the functions $\tpsi_{\xi}$ and $\psi^{(\xi)}$ do not intersect on $(1,b)$. Since $(\psi^{(\xi)})'(1)= \frac{2\xi^2 -\la^{(\xi)}}{2\xi}>0$ and $\tpsi_{\xi}'(x)\to +\infty$ as $x\to 1^{+}$, we must have $\psi^{(\xi)}<\tpsi_{\xi}$ on $(1,b)$ for the above choice of $b,p,q$. We point out that this special choice gives us the first kind of dHYM instability (see $\S$ \ref{subsec:examples-dhym-unstab-types}). The discussions above can be summarized as follows.
	
\begin{lem}\label{lem:initial-func-less-than-limit}
For the special values $q=b=3$ and $p=2bq=18$ we have $\psi^{(\xi)}\leq \tpsi_{\xi}$ on $[1,b]$. In particular, the initial function $\psi_0$ given by \eqref{eq:initial function-dim 3} satisfies $\psi_0 \leq \tpsi_\xi$ on the interval $[1,b]$.
\end{lem}

\vspace*{3mm}
\subsection{Proof of the convergence}
	
We are now in a position to prove the Theorem \ref{thm:singular-dhym-P3 blowup-cotangnet flow_Introduction}. In fact, we prove the convergence result of the cotangent flow on $X=\Bl_{x_0}\PP^3$ by generalizing the arguments of Proposition $4.11$ in \cite{DMS24}. 
	
	\begin{prop}
		\label{prop:conv-cotangent-flow-calabi-psi}
		Let $X=\Bl_{x_0}\PP^3$ with fixed K\"ahler class $\be=b[H]-[E]$ and a real $(1,1)$-class $\al=p[H]-q[E]$. Suppose $0<q \leq c_q + \sqrt{c_{q}^{2} +1}$ (i.e. $(X,\al,\be)$ is not dHYM stable). Let $\psi(x,t)$ solve the cotangent flow \eqref{eq:cotangent flow on P^n blowup at a point} with $\psi(x,0) = \psi_0(x)$ satisfies the assumption \eqref{eq:comparison and monotonicity for initial function}. Then $\psi(x,t)$ converges uniformly, and in $C^1_{loc}$ on $(1,b]$, to the function $\tilde\psi_\xi$. 
	\end{prop}
	
	\begin{proof}
		By the comparison and monotonicity from $\S$ \ref{subsec:dhym-key-lem}, the point-wise limit $\psi_\infty(x) := \lim_{t\rightarrow\infty}\psi(x,t)$ exists, and it is a lower semi-continuous function on $[1,b]$. Moreover,  $\psi_\infty\leq \tilde\psi_\xi$ on $[1,b]$ with $\psi_\infty(b) = p$, $\psi_\infty(1)=q$. Also, $x\psi_\infty(x)$ is increasing on $[1,b]$ (by the Lemma \ref{lem:supercritical equiv to positivity and monotonicity_blowup of P n} and equation \eqref{eq:upper-lower-bdd-Lag-op} below). We let $$s_0 := \lim_{x\rightarrow 1^+}\psi_\infty(x).$$ 
		Since $\theta_\om(\chi_0)<\pi$ (by the assumption \eqref{eq:comparison and monotonicity for initial function}) and this condition is preserved along the flow (cf. \cite[Lemma 3.2]{FYZ}), we get
		\begin{equation}\label{eq:upper-lower-bdd-Lag-op}
			0 < \inf_{x\in[1,b]}\theta(x,0) \leq \theta(x,t) \leq \sup_{x\in[1,b]}\theta(x,0) < \pi
		\end{equation}
		for all $(x,t)\in [1,b]\times [0,\infty)$. It implies
		\begin{equation}\label{eq:im-k-equals-n-positive-supcrit-equiv}
			\left(\im (\psi + \ii x)^{3}\right)' = 3 (\psi^2 + x^2)({\psi'}^2 +1)^{\frac{1}{2}} \sin (\theta(x,t)) >0.
		\end{equation}
		Moreover, setting $c(x,t) := \cot\theta(x,t)$, we have $$\left(\re(\psi + \ii x)^3\right)' = c(x,t) \left(\im(\psi + \ii x)^3\right)'.$$
		We now prove the proposition through several claims.
		\begin{claim}
			There exists a sequence $t_j\rightarrow \infty$ and $c_\infty\in\RR$ such that $c(x,t_j)\rightarrow c_\infty$ for each $x\in (1,b)$ with the convergence being uniform on any compact subset $K\subset (1,b)$.
		\end{claim}
		\begin{proof}
			By the flow equation \eqref{eq:cotangent flow on P^n blowup at a point} and $\dot\psi\geq 0$, we have $c'(x,t)\geq 0$, and so $c(x,t)$ is increasing in $x$. Moreover, $c(x,t)$ is a bounded function of $(x,t)$ since $\theta(x,t)$ is uniformly bounded on $(0,\pi)$. Let us now fix a $0<\delta\ll 1$ and time $T>0$. Then  $$\int_{T}^{\infty}\int_{1+\delta}^{b-\delta}\Big|\frac{\partial \psi}{\partial t}\Big| = \int_{1+\delta}^{b-\delta}\Big(\psi_\infty(x) - \psi(x,T)\Big)\,dx \leq C$$ for some uniform constant $C$ independent of $T$. Now, since $Q(x)$ is uniformly lower bounded away from zero on $[1+\delta,b-\delta]$ we see that there exists a constant $C_\delta$ independent of $T$ such that 
			\begin{align*}
				&\int_T^\infty \Big(c(b-\delta,t) - c(1+\delta,t)\Big)\,dt \\ &= \int_T^\infty\int_{1+\delta}^{b-\delta}c'(x,t)\,dx dt \\ &= \int_T^\infty\int_{1+\delta}^{b-\delta} \frac{1}{Q}\frac{\d \psi}{\d t} dx dt \\ &< C_\delta.
			\end{align*}
			So there exists a sequence $t_j\rightarrow\infty$ such that $$\lim_{j\rightarrow\infty}\Big(c(b-\delta,t_j) - c(1+\delta,t_j)\Big) = 0.$$ 
			By passing to a subsequence we may assume that $c(1+\delta,t_j)\rightarrow c_\infty$. By monotonicity we then have that $c(x,t_j)\rightarrow c_\infty$ for all $x\in [1+\delta,b-\delta]$. Then a simple diagonalisation argument implies that $c_\infty$ is independent of $\delta$ and the claim is proved.
		\end{proof}
		
		\begin{claim}
			For all $x\in (1,b)$ we have $$\re(\psi_\infty + \ii x)^3 - c_\infty \im(\psi_\infty + \ii x)^3 = A_\infty$$ for some constant $A_\infty = \re(s_0 + \ii)^3 - c_\infty \im(s_0 + \ii)^3$ where $s_0=\psi_\infty(1+)$.
		\end{claim}
		\begin{proof}
			For the sequence $\{t_j\}$ in the previous claim, we let
			\begin{align*}
				h_j(x) := \re(\psi(x,t_j) + \ii x)^3 - c_\infty \im(\psi(x,t_j) + \ii x)^3, 
			\end{align*}
			and we denote it's limit by
			$$ h(x) := \re(\psi_\infty(x) + \ii x)^3 - c_\infty \im(\psi_\infty(x) + \ii x)^3.$$
			Once again let $0<\delta\ll 1$ and take any $x\in[1+\del, b-\del]$. Then 
			\begin{align*}
				&\lvert h_j(x) - h_j(1+\del)\rvert = \left\lvert\int_{1+\del}^{x} h'_j(x) dx \right\rvert \\
				&= \left\lvert \int_{1+\del}^{x} \left(\re(\psi(x,t_j) + \ii x)^3\right)' dx -c_\infty \int_{1+\del}^{x} \left(\im(\psi(x,t_j) + \ii x)^3\right)' dx \right\rvert \\
				&\leq \int_{1+\del}^{x}  \left\lvert(c(x,t_j) - c_\infty)\right\rvert \left(\im(\psi(x,t_j) + \ii x)^3\right)' dx \\
				&\leq \sup_{x\in[1,b]}\left\lvert(c(x,t_j) - c_\infty)\right\rvert \int_{1+\del}^{x} \left(\im(\psi(x,t_j) + \ii x)^3\right)' dx \\
				&= \sup_{x\in[1,b]}\left\lvert(c(x,t_j) - c_\infty)\right\rvert \left(\im(\psi(x,t_j) + \ii x)^3 - \im(\psi(1+\del,t_j) + \ii (1+\del))^3\right) \\
				&\leq C \sup_{x\in[1,b]}\left\lvert(c(x,t_j) - c_\infty)\right\rvert,
			\end{align*} 
			for some uniform constant $C>0$ using Eq. \eqref{eq:im-k-equals-n-positive-supcrit-equiv} (or by the Lemma \ref{lem:supercritical equiv to positivity and monotonicity_blowup of P n}). Letting $j\rightarrow \infty$, it implies that $h(x) = h(1+\del)$ for all $\delta>0$ small and for all $x\in[1+\del, b-\del]$. Therefore, $h(x)=A_\infty$ for all $x\in(1,b)$.
		\end{proof}
		
		\begin{claim}
			Finally we prove that $\psi_\infty \equiv \tilde\psi_\xi$.
		\end{claim}
		\begin{proof}
			Firstly,  note that since $\psi_\infty(1+)= s_0$ and $\psi_\infty(b) = p$, we see that $\psi_\infty \equiv\tilde\psi_{s_0}.$ Next, since $x\psi_\infty$ is increasing,  we have $s_{0}\geq \xi$. Next, by Lemma \ref{lem:comparison-to-limit-gen dim} we have that $\psi_\infty\leq \tilde\psi_\xi$. On the other hand, by Lemma \ref{lem:monotonicity-psi-s-any-dim}, if $s_{0}>\xi$, then $\psi_\infty = \tilde\psi_{s_0} > \tilde\psi_\xi$ on $[1,b)$, a contradiction. Hence $s_{0} = \xi$ and this completes the proof of the proposition.
		\end{proof}
	\end{proof}
	
We remark that the arguments of Proposition \ref{prop:conv-cotangent-flow-calabi-psi} can also be generalized for higher dimensions $n\geq 4$ as mentioned in \cite{Met-PhD-thesis} provided the assumption \eqref{eq:comparison and monotonicity for initial function} holds and there exists a unique branch $\tpsi_{\xi}$. Moreover, the higher order estimates can be proved in a similar manner to those for the J-flow (see \cite{SW08} and \cite{DMS24}) using the estimates from \cite{CJY, Tak-TLPF, FYZ}. This completes the proof of Theorem \ref{thm:singular-dhym-P3 blowup-cotangnet flow_Introduction} (note that the part $(1)$ of Theorem \ref{thm:singular-dhym-P3 blowup-cotangnet flow_Introduction} follows from \cite{Chu-Lee-Tak} and \cite{FYZ}).

\vspace*{5mm}
\section*{Acknowledgements}
The author would like to thank his Ph.D. supervisor Dr. Ved V. Datar for helpful discussions, continuous encouragement and support. The author also would like to thank Prof. Jian Song for many interesting discussions. This work was supported in part by a Ph.D. fellowship from the Indian Institute of Science (IISc), Bengaluru. The author is grateful to The Institute of Mathematical Sciences (IMSc), Chennai where the final draft of this work was prepared. The author also thanks the anonymous referees for their valuable suggestions and remarks which improved the presentation of this article. The author would like to express his gratitude to the Indian Institute of Technology (IIT) Bombay where this article has been revised in the current version.

\vspace*{5mm}
	

	

\begin{thebibliography}{100}
		
\bibitem{Bal23} Ballal, A.: {\em The supercritical deformed Hermitian Yang-Mills equation on compact projective manifolds}, {\sc Illinois J. Math.} {\bf 67} (2023), no. 1, 73--99.
		
\bibitem{BHPV04-book} Barth, W.P., Hulek, K., Peters, C.A.M. and Van de Ven, A.: {\em Compact complex surfaces}. {\sc Ergeb. Math. Grenzgeb. (3), 4.}, A Series of Modern Surveys in Mathematics, Springer-Verlag, Berlin, 2004, xii+436 pp. ISBN: 3-540-00832-2
		
\bibitem{BedTay87} Bedford, E. and Taylor, B.A.: {\em Fine topology, Shilov boundary, and $(dd^{c})^n$}, {\sc J. Funct. Anal.} {\bf 72} (1987), no. 2, 225--251.
		
\bibitem{BEGZ10} Boucksom, S., Eyssidieux, P., Guedj, V. and Zeriahi A.: {\em Monge-Amp\`{e}re equations in big cohomology classes}, {\sc Acta Math.} {\bf 205} (2010), no. 2, 199--262.
		
\bibitem{Cal-extremal} Calabi, E.: {\em Extremal K\"ahler metrics}, {\sc Seminar on Differential Geometry, Vol. {\bf 102} of Ann. Math. Studies}, Princeton Univ. Press, Princeton, N.J. (1982), 259--290.
		
\bibitem{ChaJac23} Chan, Y.H. and Jacob, A.: {\em Singularity formation along the line bundle mean curvature flow}, Int. Math. Res. Not. IMRN 2025, no. 5, Paper No. rnaf037, 20 pp.
		
\bibitem{gchen} Chen, G.: {\em The J-equation and the supercritical deformed Hermitian-Yang-Mills equation}, {\sc Invent. Math.} {\bf 225} (2021), no. 2, 529--602.
		
\bibitem{CJY} Collins, T.C., Jacob, A. and Yau, S.-T.: {\em $(1, 1)$ forms with specified Lagrangian phase: a priori estimates and algebraic obstructions}, {\sc Camb. J. Math.} {\bf 8} (2020), no. 2, 407--452.
		
\bibitem{Chu-Lee-Tak} Chu, J., Lee, M.C. and Takahashi, R.: {\em A Nakai-Moishezon type criterion for supercritical deformed Hermitian-Yang-Mills equation}, {\sc J. Differential Geom.} {\bf 126} (2024), no. 2, 583--632.
		
\bibitem{CXY18} Collins, T.C., Xie, D. and Yau, S.-T.: {\em The deformed Hermitian-Yang-Mills equation in geometry and physics}, Oxford University Press, Oxford, 2018, 69--90. ISBN: 978-0-19-880201-3; 978-0-19-880200-6
		
\bibitem{DMS24} Datar, V.V., Mete, R. and Song, J.: {\em Minimal slopes and bubbling for complex Hessian equations}, {\sc Adv. Math.} {\bf 491} (2026), Paper No. 110865, 84 pp.
		
\bibitem{DatPin2021} Datar, V.V. and Pingali, V.P.: {\em A numerical criterion for generalised Monge-Amp\`{e}re equations on projective manifolds}, {\sc Geom. Funct. Anal.} {\bf 31} (2021), no. 4, 767--814.
		
\bibitem{FYZ} Fu, J., Yau, S.-T. and Zhang, D.: {\em A new flow solving the LYZ equation in K\"ahler geometry}, {\sc J. Differential Geom.} {\bf 128} (2024), no. 1, 153--192.
		
\bibitem{JacSh} Jacob, A. and Sheu, N.: {\em The deformed Hermitian-Yang-Mills equation on the blowup of $\PP^n$}, {\sc Asian J. Math.} {\bf 26} (2022), no. 6, 847--864.
		
\bibitem{JY} Jacob, A. and Yau, S.-T.: {\em A special Lagrangian type equation for holomorphic line bundles}, {\sc Math. Ann.} {\bf 369} (2017), no. 1-2, 869--898.
		
\bibitem{Laz-bk-v1} Lazarsfeld, R.: {\em Positivity in Algebraic Geometry I, classical setting: line bundles and linear series}, {\sc Springer Berlin, Heidelberg} 2004, pp. XVIII, 387, https://doi.org/10.1007/978-3-642-18808-4.
		
\bibitem{LYZ01} Leung, N.C., Yau, S.-T. and Zaslow, E.: {\em From special Lagrangian to Hermitian-Yang-Mills via Fourier-Mukai transform}, {\sc AMS/IP Stud. Adv. Math.}, 23, Amer. Math. Soc., Providence, RI, 2001, 209--225.
		
\bibitem{MMMG00} Mari\~no, M., Minasian, R., Moore, G. and Strominger, A.: {\em Nonlinear instantons from supersymmetric p-branes}, {\sc J. High Energy Phys.} 2000, no. 1, Paper 5, 32 pp.
		
\bibitem{Met-PhD-thesis} Mete, R.: {\em On two complex Hessian equations and convergence of corresponding flows}, PhD thesis, Indian Institute of Science, Bengaluru,
2024.
		
\bibitem{Pingali-dHYM 3 fold} Pingali, V.P.: {\em The deformed Hermitian Yang-Mills equation on three-folds}, {\sc Anal. PDE}, {\bf 15} (2022), no. 4, 921--935.
		
\bibitem{Song2020} Song, J.: {\em Nakai-Moishezon criterions for complex Hessian equations}, \url{https://arxiv.org/abs/2012.07956}.
		
\bibitem{SW08} Song, J. and Weinkove, B.: {\em On the convergence and singularities of the J-flow with applications to the Mabuchi energy}, {\sc Comm. Pure Appl. Math.} {\bf 61} (2008), no. 2, 210--229. 
		
\bibitem{Tak-TLPF} Takahashi, R.: {\em Tan-concavity property for Lagrangian phase operators and applications to the tangent Lagrangian phase flow}, {\sc Internat. J. Math.} {\bf 31} (2020), no. 14, 2050116, 26 pp.
		
\end{thebibliography}
\end{document}